\documentclass[leqno,12pt]{amsart} 
\usepackage{color}
\usepackage{amssymb}
\theoremstyle{plain} 
\newtheorem{theorem}{\indent\sc Theorem}[section] 
\newtheorem{lemma}[theorem]{\indent\sc Lemma}

\newtheorem{proposition}[theorem]{\indent\sc Proposition}

\theoremstyle{definition} 

\newtheorem{remark}[theorem]{\indent\sc Remark}

\numberwithin{equation}{section}

\begin{document}
\allowdisplaybreaks
\title[Global existence for a system of multiple-speed wave equations]
{Global existence for a system of multiple-speed wave equations 
violating the null condition} 


\author[K. Hidano]{Kunio Hidano$^*$}
\author[K. Yokoyama]{Kazuyoshi Yokoyama}
\author[D. Zha]{Dongbing Zha$^\dag$}

\renewcommand{\thefootnote}{\fnsymbol{footnote}}
\footnote[0]{2010\textit{ Mathematics Subject Classification}.
 Primary 35L52, 35L15; Secondary 35L72}

\thanks{ 
$^*$Partly supported by 
the Grant-in-Aid for Scientific Research (C) (No.\,18K03365), 
Japan Society for the Promotion of Science (JSPS) }
\thanks{
$^\dag$Supported by National Natural Science Foundation of 
China (No.11801068) and 
Fundamental Research Funds for 
the Central Universities (No.\,2232021G-13)
}

\address{ 
Department of Mathematics \endgraf
Faculty of Education \endgraf
Mie University \endgraf
1577 Kurima-machiya-cho Tsu \endgraf
Mie Prefecture 514-8507 \endgraf
Japan
}
\email{hidano@edu.mie-u.ac.jp}

\address{ 
Hokkaido University of Science \endgraf
7-Jo 15-4-1 Maeda, Teine, Sapporo \endgraf
Hokkaido 006-8585 \endgraf
Japan
}
\email{yokoyama@hus.ac.jp}

\address{ 
Department of Mathematics and \endgraf
Institute for Nonlinear Sciences \endgraf 
Donghua University \endgraf
Shanghai 201620 \endgraf
PR China
}
\email{ZhaDongbing@163.com}


\maketitle

\begin{abstract}
We discuss the Cauchy problem 
for a system of semilinear wave equations 
in three space dimensions with 
multiple wave speeds. 
Though our system does not satisfy the standard null condition, 
we show that 
it admits a unique global solution 
for any small and smooth data. 
This generalizes a preceding result due to 
Pusateri and Shatah. 

The proof is carried out by the energy method 
involving a collection of generalized derivatives. 
The multiple wave speeds disable the use of 
the Lorentz boost operators, 
and our proof therefore relies upon the version of 
Klainerman and Sideris. 
Due to the presence of nonlinear terms 
violating the standard null condition, 
some of components of the solution 
may have a weaker decay as $t\to\infty$, 
which makes it difficult even to establish 
a mildly growing (in time) bound for the high energy estimate. 
We overcome this difficulty by 
relying upon the ghost weight energy estimate of Alinhac and 
the Keel-Smith-Sogge type $L^2$ weighted space-time estimate for 
derivatives.
\end{abstract}

\section{Introduction}
This paper is concerned with the Cauchy problem for 
a system of semilinear wave equations 
in three space dimensions of the form
\begin{equation}\label{eq1}
\begin{cases}
\displaystyle{
\partial_t^2 u_1-\Delta u_1
=
F_1(\partial u_1,\partial u_2,\partial u_3),
\,\,t>0,\,x\in{\mathbb R}^3
             },\\
\displaystyle{
\partial_t^2 u_2-\Delta u_2
=
F_2(\partial u_1,\partial u_2,\partial u_3),
\,\,t>0,\,x\in{\mathbb R}^3
              },\\
\displaystyle{
\partial_t^2 u_3-c_0^2\Delta u_3
=
F_3(\partial u_1,\partial u_2,\partial u_3),
\,\,t>0,\,x\in{\mathbb R}^3
              }
\end{cases}
\end{equation}
subject to the initial condition
\begin{equation}\label{data}
(u_i(0),\partial_t u_i(0))
=
(f_i,g_i)
\in
C_0^\infty({\mathbb R}^3)\times C_0^\infty({\mathbb R}^3),\,\,
i=1,2,3,
\end{equation}
where $(u_1,u_2,u_3):(0,\infty)\times{\mathbb R}^3\to{\mathbb R}^3$, 
$\partial=(\partial_0,\partial_1,\partial_2,\partial_3)$, 
$\partial_0=\partial/\partial t$, 
$\partial_i=\partial/\partial x_i$, 
and $c_0>0$. 
Moreover, $F_1(y)$, $F_2(y)$, and $F_3(y)$ 
are polynomials in $y\in{\mathbb R}^{12}$ 
of degree $\geq 2$. That is, 
we suppose that 
the nonlinear term has the form 
\begin{equation}\label{formF}
F_i(\partial u_1,\partial u_2,\partial u_3)
=
F_i^{jk,\alpha\beta}(\partial_\alpha u_j)(\partial_\beta u_k)
+
C_i(\partial u_1,\partial u_2,\partial u_3),\,\,i=1,2,3,
\end{equation}
where $C_i(y)$ is a polynomial in $y\in{\mathbb R}^{12}$ 
of degree $\geq 3$. 
In what follows, we suppose $F_i^{jk,\alpha\beta}=0$ if $j>k$, 
without loss of generality. 
Here, and in the following discussion as well, 
we use the summation convention: 
if lowered and uppered, 
repeated Greek letters and Roman letters 
are summed for $0$ to $3$ and 
$1$ to $3$, respectively. 

Though our main interest lies in global existence 
of small, smooth solutions in the case $c_0\ne 1$, 
we first review some of the results for the case $c_0=1$. 
It follows from the fundamental result of John and Klainerman \cite{JK} 
that the equation (\ref{eq1}) admits a unique 
``almost global'' solution for small, smooth data with compact support. 
That is, the time interval on which 
the local solution exists becomes exponentially large 
as the size of initial data gets smaller and smaller. 
Almost global existence is the most that one can expect in general. 
Indeed, nonexistence of global solutions is known even for 
small data. 
See, e.g., John \cite{John1981} and Sideris \cite{Sideris1983} 
for the scalar equations $\partial_t^2 u-\Delta u=(\partial_t u)^2$ 
and $\partial_t^2 u-\Delta u=|\nabla u|^2$, 
respectively. 
On the other hand, if the null condition is satisfied, 
that is, for any given $i,j,k$ we have 
$F_i^{jk,\alpha\beta}X_\alpha X_\beta\equiv 0$ 
for all $(X_0,X_1,X_2,X_3)\in{\mathbb R}^4$ satisfying 
$X_0^2=X_1^2+X_2^2+X_3^2$, 
then it follows from the seminal result of 
Christodoulou \cite{Christodoulou} and 
Klainerman \cite{KlainermanNull86} 
(see also Alinhac \cite[p.\,94]{Al2010} 
for a new proof using $L^2$ space-time weighted estimates 
for some special derivatives) 
that the equation (\ref{eq1}) admits a unique global 
solution for small, smooth data. 
Christodoulou employed the method of conformal mapping 
and Klainerman employed the energy method involving 
the generators of the translations, 
the Lorentz transformations, and the dilations. 

Let us turn our attention to the case $c_0\ne 1$, 
which does not seem amenable to the method in \cite{Christodoulou} 
or \cite{KlainermanNull86} 
because of the presence of multiple wave speeds. 
Alternative techniques based on 
a smaller collection of generators 
have been explored by a lot of authors, 
such as Kovalyov \cite{Kov} and Yokoyama \cite{Y} 
using point-wise estimations of the fundamental solution, 
Klainerman and Sideris \cite{KS} and Sideris and Tu \cite{SiderisTu} 
without relying upon point-wise estimations of the fundamental solution, 
and Keel, Smith and Sogge \cite{KSS2002} 
using $L^2$ space-time weighted estimates for derivatives. 
Obviously, the technique in \cite{KSS2002} is applicable to 
the Cauchy problem (\ref{eq1})--(\ref{data}) with $c_0\ne 1$ 
and leads to almost global existence result. 
Moreover, if $c_0\ne 1$ and the null condition in the sense of 
\cite{Y}, \cite{SiderisTu}, and \cite{LNS} is satisfied, 
that is, we have for any $i=1,2$ and $(j,k)=(1,1)$, $(1,2)$, 
and $(2,2)$
\begin{align}
&F_i^{jk,\alpha\beta}X_\alpha X_\beta\equiv 0,\,\,
X\in{\mathcal N}^{(1)},\label{null1}\\
&F_3^{33,\alpha\beta}X_\alpha X_\beta\equiv 0,\,\,
X\in{\mathcal N}^{(c_0)},\label{null2}
\end{align}
then it follows from \cite{Y}, \cite[Remark following Theorem 3.1]{SiderisTu}, and 
\cite[Theorem 1.1]{LNS} 
that the equation (\ref{eq1}) admits a unique global solution 
for small, smooth data. 
(We note that as pointed out in \cite{H2004}, 
the argument of Sideris and Tu is general enough to 
handle the nonlinear terms satisfying (\ref{null1})--(\ref{null2}), 
although they were not explicitly treated in \cite{SiderisTu}.) 
Here, and in the following as well, 
we use the notation 
\begin{equation}\label{definitionnc}
{\mathcal N}^{(c)}:=
\{
X=(X_0,X_1,X_2,X_3)\in{\mathbb R}^4\,:\,
X_0^2=c^2(X_1^2+X_2^2+X_3^2)
\},\,\,c>0.
\end{equation}
Recently, there have been a lot of activities 
in studying systems of wave equations with wider classes of 
quadratic nonlinear terms 
for which one still enjoys global solutions 
for any small, smooth data. 
See, e.g., \cite{LR2003}, \cite{Al2006}, \cite{Lind2008}, 
\cite{KMS}, \cite{HY2017}, and \cite{Keir} for 
systems in three space dimensions with equal propagation speeds. 
As for (\ref{eq1}) with $c_0\ne 1$, 
we easily see that 
the condition (\ref{null1})--(\ref{null2}) is sufficient but not necessary 
for global existence. Indeed, 
setting 
$F_1=(\partial_t u_2)^2$, 
$F_2=F_3=(\partial_t u_2)(\partial_t u_3)$, 
we see this term $(\partial_t u_2)^2$ violating the condition 
(\ref{null1}) but we obtain global solutions by first 
solving the system consisting of the second and the third equations in 
(\ref{eq1}) on the basis of the results in 
\cite{Y}, \cite{SiderisTu}, and \cite{LNS}
and then regarding the first equation in (\ref{eq1}) 
just as the inhomogeneous wave equation with the 
``source term'' $(\partial_t u_2)^2$. 
Interestingly, 
using the space-time resonance method, 
Pusateri and Shatah \cite{PS} have proved that 
global existence of small solutions carries over to 
3-component systems with a class of nonlinear terms, say, 
$F_1=(\partial_t u_2)^2+(\partial_t u_1)(\partial_t u_3)^2$, 
$F_2=(\partial_t u_2)(\partial_t u_3)+
\bigl((\partial_t u_2)^2-|\nabla u_2|^2\bigr)$, 
$F_3=(\partial_t u_2)(\partial_t u_3)+(\partial_t u_1)(\partial_t u_2)^2$. 
They also mention that 
$\partial u_1$ has a weaker decay as $t\to\infty$. 
Inspired by their observation, 
we like to find 3-component and 2-speed systems with 
a wider class of nonlinear terms 
for which one still obtains global solutions for small, smooth data. 
In particular, we are interested in the case where 
$u_1$, which may have a weaker decay, is involved 
in quadratic nonlinear terms. 
We suppose 
\begin{align}
&F_1^{11,\alpha\beta}X_\alpha X_\beta\equiv 0,\,\,
X\in{\mathcal N}^{(1)},\label{assumption1}\\
&F_2^{11,\alpha\beta}X_\alpha X_\beta
=
F_2^{12,\alpha\beta}X_\alpha X_\beta
=
F_2^{22,\alpha\beta}X_\alpha X_\beta\equiv 0,\,\,
X\in{\mathcal N}^{(1)},\label{assumption2}\\
&
F_3^{33,\alpha\beta}X_\alpha X_\beta\equiv 0,\,\,
X\in{\mathcal N}^{(c_0)},\label{assumption3}\\
&
F_2^{13,\alpha\beta}=0
\mbox{\,\,for any $\alpha$, $\beta$},\label{assumption4}\\
&
F_3^{13,\alpha\beta}=0
\mbox{\,\,for any $\alpha$, $\beta$,}\label{assumption5}\\
&
F_3^{11,\alpha\beta}X_\alpha X_\beta
=F_3^{12,\alpha\beta}X_\alpha X_\beta\equiv 0,
\,\,X\in{\mathcal N}^{(1)},\label{assumption6}
\end{align}
which means that 
since the condition (\ref{assumption1}) is weaker than 
(\ref{null1}) with $i=1$, 
the nonlinear term such as 
\begin{equation}\label{modelterm2019502}
F_1=(\partial_t u_2)^2
+
(\partial_t u_1)(\partial_t u_2)
+
\bigl(
(\partial_t u_1)^2-|\nabla u_1|^2
\bigr)
+
C_1(\partial u_1,\partial u_2,\partial u_3)
\end{equation}
is admissible. Also, any cubic term is admissible. 
On the other hand, we need the restrictive conditions 
(\ref{assumption4})--(\ref{assumption6}) 
in order to obtain a mildly growing bound 
for the high energy estimate of $u_2$, $u_3$, 
though readers might expect to benefit from 
difference of propagation speeds. 
Before stating the main theorem, 
we set the notation. 
We use the operators 
$\Omega_{jk}:=x_j\partial_k-x_k\partial_j$, 
$1\leq j<k\leq 3$ and 
$S:=t\partial_t+x\cdot\nabla$. 
The operators 
$\partial_1$, $\partial_2$, $\partial_3$, 
$\Omega_{12}$, $\Omega_{23}$, $\Omega_{13}$ 
and $S$ are denoted by 
$Z_1$, $Z_2,\dots,Z_7$, respectively. 
For multi-indices $a=(a_1,a_2,\dots,a_7)$, 
we set $Z^a:=Z_1^{a_1}Z_2^{a_2}\cdots Z_7^{a_7}$. 
Setting 
$$
E(v(t);c):=
\frac12
\int_{{\mathbb R}^3}
\bigl(
(\partial_t v(t,x))^2
+
c^2
|\nabla v(t,x)|^2
\bigr)dx
$$
for $c>0$, we define
\begin{equation}
N_{\kappa}(v(t);c)
:=
\biggl(
\sum_{|a|\leq\kappa-1}
E(Z^a v(t);c)
\biggr)^{1/2},\,\,
\kappa\in{\mathbb N}.
\end{equation}
When there is no confusion, 
we abbreviate $N_{\kappa}(v(t);c)$ to $N_{\kappa}(v(t))$. 
To measure the size of data 
$(f,g)$ with $f=(f_1,f_2,f_3)$ and $g=(g_1,g_2,g_3)$, 
we use 
\begin{equation}
\|(f,g)\|_D
:=
\sum_{i=1}^3
\biggl(
\sum_{|a|=1}^4
\|
\langle x\rangle^{|a|-1}\partial_x^a f_i
\|_{L^2({\mathbb R}^3)}
+
\sum_{|a|=0}^3
\|
\langle x\rangle^{|a|}\partial_x^a g_i
\|_{L^2({\mathbb R}^3)}
\biggr).
\end{equation}
We are in a position to state the main theorem.
\begin{theorem}\label{ourmaintheorem}
Suppose $c_0\ne 1$ in \mbox{$(\ref{eq1})$} 
and suppose \mbox{$(\ref{assumption1})$--$(\ref{assumption6})$}. 
There exists an $\varepsilon_0>0$ such that 
if the initial data 
$(f_i,g_i)\in C_0^\infty({\mathbb R}^3)\times C_0^\infty({\mathbb R}^3)$ 
$(i=1,2,3)$ satisfy 
$\|(f,g)\|_D<\varepsilon_0$, 
then the Cauchy problem \mbox{$(\ref{eq1})$--$(\ref{data})$} admits 
a unique global solution satisfying 
\begin{align}
&
N_3(u_1(t))
\leq
C\varepsilon_0(1+t)^\delta,\,\,
N_4(u_1(t))
\leq
C\varepsilon_0(1+t)^{2\delta},\label{growth1}\\
&
N_3(u_2(t)),\,N_3(u_3(t))
\leq
C\varepsilon_0,\,\,
N_4(u_2(t)),\,N_4(u_3(t))
\leq
C\varepsilon_0(1+t)^\delta.\label{growth2}
\end{align}
Here $\delta$ is a small constant such that 
$0<\delta<1/24$.
\end{theorem}
\begin{remark}
Using (\ref{ellinfty}), (\ref{2019mninequality}) with $\mu=3$, and (\ref{j3}), 
we see that the solution $(u_1,u_2,u_3)$ in 
Theorem \ref{ourmaintheorem} satisfies
\begin{equation}
\|\partial u_1(t)\|_{L^\infty({\mathbb R}^3)}
=
O(t^{-1+\delta}),
\,\,
\|\partial u_i(t)\|_{L^\infty({\mathbb R}^3)}
=
O(t^{-1}),\,i=2,3
\end{equation}
as $t\to\infty$.
\end{remark}
\begin{remark}
Suppose that the system (\ref{eq1}) satisfies the assumptions 
(\ref{assumption1})--(\ref{assumption6}) in Theorem \ref{ourmaintheorem}. 
The referee has kindly pointed out that 
if we ignore the third line of (\ref{eq1}) and 
remove $u_3$ from the first two lines, 
then the remaining 2-component system satisfies 
the weak null condition in Alinhac 
\cite[see (AA)--(${\overline{\rm AA}}$)]{Al2006}. 
Also, if we ignore the first line of (\ref{eq1}) and 
remove $u_1$ from the last two lines, 
then the remaining 2-component system satisfies 
the null condition in Yokoyama \cite{Y} and 
Sideris-Tu \cite{SiderisTu}. 
The assumptions (\ref{assumption1})--(\ref{assumption6}) 
are considerably weaker than those in \cite{PS} indeed, 
but they still seem restrictive. 
There arises a natural question to what extent 
we can weaken (\ref{assumption1})--(\ref{assumption6}). 
In this regard, one might want to ask whether or not 
the 3-component system (\ref{eq1}) admit a unique global solution 
for small, smooth data, 
if Alinhac's conditions 
(AA) and (${\overline{\rm AA}}$) are satisfied 
by the 2-component system 
derived by removal of $u_3$ from (\ref{eq1}) 
and 
the null condition in \cite{Y} and \cite{SiderisTu} is satisfied 
by the 2-component system 
derived by removal of $u_1$ from (\ref{eq1}). 
This seems to the authors an interesting open question. 
\end{remark}
Differently from the space-time resonance method of 
Pusateri and Shatah \cite{PS}, 
the proof of our main theorem employs the method of 
Klainerman and Sideris \cite{KS} 
which is the energy method involving 
the generators of the translations, 
the spatial rotations, 
and the dilations. 
It does not involve the generators of the 
hyperbolic rotations, 
and has successfully led to results of 
global existence of small solutions 
under the null condition, 
for systems of multiple-speed wave equations 
\cite{SiderisTu}, 
and for the equation of elasticity \cite{Sideris1996}, \cite{Sideris2000}. 
Unlike the system considered in Sideris and Tu \cite{SiderisTu}, 
the system (\ref{eq1}) is permitted to involve the term $(\partial_t u_2)^2$ 
or $(\partial_t u_1)(\partial_t u_2)$ 
in the first equation (see (\ref{modelterm2019502}) above), 
and the presence of terms violating the null condition 
causes a weaker decay of $\partial u_1$ as $t\to\infty$. 
Therefore, we must enhance the discussion in \cite{SiderisTu}, 
although we basically follow their argument 
based on the two-energy method. 
We recall that the proof of global existence in \cite{SiderisTu} 
employed the ``high energy'' estimate 
and the ``low energy'' estimate, 
allowing the bound in the former estimate to grow mildly in time, 
and establishing the uniform (in time) bound in the latter estimate 
by virtue of the null condition and the difference of 
propagation speeds. 
We note that 
because of the problem of ``loss of derivatives'' 
caused by the use of the standard estimation lemma 
for the null forms (see \cite[Lemma 5.1]{SiderisTu}), 
it is only for the estimate of 
the {\it low} energy that 
the null condition plays a role in \cite{SiderisTu}. 
In the present case, 
owing to the weaker decay of $\partial u_1$, 
even a mildly growing bound in the high energy estimate is far from trivial. 
A similar difficulty already occurred in the proof of 
Alinhac \cite{Al2001} for global existence of small solutions 
to the null-form quasilinear (scalar) wave equations in {\it two} space dimensions. 
(Recall that the time decay rate of solutions in two space dimensions is worse 
than in three space dimensions.) 
Creating the ghost weight energy method, 
he succeeded in employing the null condition 
for the purpose of establishing 
a mildly growing bound in the {\it high} energy estimate. 
(See also \cite{ZhaDCDS} for this matter.) 
Alinhac set up his remarkable method by relying upon the generators of the hyperbolic rotations, and we note that  
his technique, combined with the method of Klainerman and Sideris, 
remains useful without such operators. 
See \cite{ZhaDCDS}, \cite{Zha}, and \cite{HZ2019}. 
In order to obtain such an estimate 
for the high energy, 
we can therefore rely upon the ghost weight technique 
and utilize a certain $L^2$ space-time weighted norm 
for the special derivatives $\partial_j u_1+(x_j/|x|)\partial_t u_1$ 
along with the estimation lemma (see Lemma \ref{estimationlemma} below), 
when handling such a null-form nonlinear term as 
$\bigl(
(\partial_t u_1)^2-|\nabla u_1|^2
\bigr)$ 
(see (\ref{modelterm2019502}) above) 
on the region ``far from the origin'', that is, 
$\{x\in{\mathbb R}^3\,:\;|x|>(1+t)/2\}$. 

Actually, this way of handling the null-form nonlinear term 
$\bigl(
(\partial_t u_1)^2-|\nabla u_1|^2
\bigr)$ 
is effective only on the region ``far from the origin'', 
because in the present paper, 
the $L^2$ space-time weighted norm for the special derivatives 
is employed in combination with the trace-type inequality 
with the weight $r^{1-\eta}\langle t-r\rangle^{(1/2)+\eta}$ 
(see (\ref{interp}) with $\theta=(1/2)-\eta$ below, here $\eta>0$ is small enough) 
and the factor $r^{1-\eta}$ no longer yields 
the decay factor $t^{-1+\eta}$ on the region ``inside the cone'' 
$\{x\in{\mathbb R}^3\,:\;|x|<(1+t)/2\}$. 
As in \cite{SiderisTu}, inside the cone we therefore give up benefiting from 
the special structure that the null-form nonlinear terms enjoy, 
and we regard them simply as products of the derivatives, 
when considering the high energy estimate of $u_1$. 
Because of the growth of the bound even in the low energy estimate 
for $u_1$, we then proceed differently from \cite{SiderisTu}. 
Namely, we make use of the Keel-Smith-Sogge type $L^2$ weighted norm 
for usual derivatives (see Lemma \ref{ksstype} below) 
together with the trace-type inequality with weight 
$r^{1/2}\langle t-r\rangle$ 
(see (\ref{interp}) with $\theta=0$ below). 
See, e.g., (\ref{highenergyj112019july25}) below. 
In this way, such a null-form nonlinear term as 
$\bigl(
(\partial_t u_1)^2-|\nabla u_1|^2
\bigr)$ 
is no longer the hurdle 
to establishing a mildly growing bound in the high energy estimate 
of $u_1$. 

Because of the weaker decay of $\partial u_1$ 
and the mildly growing bound in the high energy estimate 
of $u_2$ (see (\ref{growth2}) above), 
the presence of such a term as 
$(\partial_t u_1)(\partial_t u_2)$ 
also causes another difficulty in establishing 
a mildly growing bound in the energy estimate of $u_1$. 
This is the reason why we use different growth rates 
for the high energy and the low energy of $u_1$ 
(see the factors $(1+t)^{2\delta}$ and $(1+t)^\delta$ in 
(\ref{growth1}) above) 
for the purpose of closing the argument. 
See (\ref{2019july191626})--(\ref{2019july191656}) below. 

This paper is organized as follows. 
In the next section, we first recall some special properties that 
the null-form nonlinear terms enjoy, 
and then we recall several key inequalities 
that play an important role in our arguments. 
Section \ref{sectionmmm} is devoted to obtaining bounds for 
certain weighted $L^2({\mathbb R}^3)$-norms of 
the second or higher-order derivatives of solutions. 
We carry out the energy estimate 
and the $L^2$ weighted space-time estimate 
in Sections \ref{sectionenergy} and \ref{l2weighted}, 
using the ghost weight method of Alinhac 
and the Keel-Smith-Sogge type estimate, 
respectively. 
In the final section, we complete the proof of 
Theorem \ref{ourmaintheorem} by using the method of continuity.

{\it Acknowledgments.} The problem of global existence for 
systems of multiple-speed wave equations 
violating the standard null condition was 
suggested by Thomas C.\,Sideris at Tohoku University in July, 2017, 
for which the authors are very grateful to him. 
Special thanks also go to the referee for a valuable comment concerning 
the weak null condition in \cite{Al2006}.
\section{Preliminaries}
We need the commutation relations. 
Let $[\cdot,\cdot]$ be the commutator: 
$[A,B]:=AB-BA$. 
It is easy to verify that 
\begin{eqnarray}
& &
[Z_i,\Box_c]=0\,\,\,\mbox{for $i=1,\dots,6$},\,\,\,
[S,\Box_c]=-2\Box_c,
\label{eqn:comm1}\\
& &
[Z_j,Z_k]=\sum_{i=1}^7 C^{j,k}_i Z_i,\,\,\,
j,\,k=1,\dots,7,
\label{eqn:comm2}\\
& &
[Z_j,\partial_k]
=
\sum_{i=1}^3 C^{j,k}_i\partial_i,\,\,\,j=1,\dots,7,\,\,k=1,2,3,
\label{eqn:comm3}\\
& &
[Z_j,\partial_t]=0,\,j=1,\dots,6,\quad [S,\partial_t]=-\partial_t
\label{eqn:comm4}.
\end{eqnarray}
Here $\Box_c:=\partial_t^2-c^2\Delta$, and 
$C^{j,k}_i$ denotes a constant depending on 
$i$, $j$, and $k$.

The next lemma states that the null form is preserved 
under the differentiation. 
Recall the definition of ${\mathcal N}^{(c)}$ 
(see (\ref{definitionnc})). 
\begin{lemma}\label{nullpreserved}
Let $c>0$. Suppose that $\{H^{\alpha\beta}\}$ satisfies 
\begin{equation}\label{nullformlemma}
H^{\alpha\beta}X_\alpha X_\beta=0
\mbox{\,\,for any $X\in{\mathcal N}^{(c)}$}. 
\end{equation}
For any $Z_i$ $(i=1,\dots,7)$, the equality 
\begin{align}
Z_i&\bigl(H^{\alpha\beta}
(\partial_\alpha v)(\partial_\beta w)\bigr)\\
&=
H^{\alpha\beta}
(\partial_\alpha Z_i v)
(\partial_\beta w)
+
H^{\alpha\beta}
(\partial_\alpha v)
(\partial_\beta Z_i w)
+
{\tilde H}_i^{\alpha\beta}
(\partial_\alpha v)
(\partial_\beta w)\nonumber
\end{align}
holds with the new coefficients 
$\{{\tilde H}_i^{\alpha\beta}\}$ 
also satisfying $(\ref{nullformlemma})$. 
\end{lemma}
See, e.g., \cite[pp.\,91--92]{Al2010} for the proof. 
It is possible to show the following lemma 
essentially in the same way as in \cite[pp.\,90--91]{Al2010}. 
\begin{lemma}\label{estimationlemma}
Suppose that 
$\{H^{\alpha\beta}\}$ satisfies $(\ref{nullformlemma})$ 
for some $c>0$. 
With the same $c$ as in $(\ref{nullformlemma})$, 
we have for smooth functions $v(t,x)$ and $w(t,x)$
\begin{equation}\label{htd20190718}
|
H^{\alpha\beta}
(\partial_\alpha v)
(\partial_\beta w)
|
\leq
C
\bigl(
|T^{(c)} v|
|\partial w|
+
|\partial v|
|T^{(c)} w|
\bigr).
\end{equation}
\end{lemma}
Here, and in the following, we use the notation 
\begin{equation}\label{2019Nov16OK}
|T^{(c)}v|
:=
\biggl(
\sum_{k=1}^3
|T_k^{(c)} v|^2
\biggr)^{1/2},\quad
T^{(c)}_k:=c\partial_k+(x_k/|x|)\partial_t.
\end{equation}
Together with (\ref{htd20190718}), 
we will later exploit the fact that 
for local solutions $u$, 
the special derivatives $T^{(c)}_i u$ have 
better space-time $L^2$ integrability, 
in addition to 
improved time decay property of their $L^\infty({\mathbb R}^3)$ norms 
as shown in the following lemma. 
\begin{lemma}[Lemma 2.2 of \cite{Zha}]\label{lemma22ofzha}
Let $c>0$. The inequality
\begin{equation}
|T^{(c)}v(t,x)|
\leq
C
\langle t\rangle^{-1}
\biggl(
|\partial_t v(t,x)|
+
\sum_{i=1}^7
|Z_i v(t,x)|
+
\langle ct-r\rangle|\partial_x v(t,x)|
\biggr)
\end{equation}
holds for smooth functions $v(t,x)$. 
\end{lemma}
Lemma \ref{lemma22ofzha} is a direct consequence of 
the identity such as 
\begin{equation}
T^{(c)}_1
=
\frac{1}{t}
\biggl(
\frac{x_1}{|x|}S
-
\frac{x_2}{|x|}\Omega_{12}
-
\frac{x_3}{|x|}\Omega_{13}
+
(ct-r)\partial_1
\biggr).
\end{equation}
The following lemma is concerned with Sobolev-type 
or trace-type inequalities. 
With $c>0$, the auxiliary norms 
\begin{align}
M_2(v(t);c)
&=
\sum_{{0\leq\delta\leq 3}\atop{1\leq j\leq 3}}
\|\langle ct-|x|\rangle
\partial_{\delta j}^2 v(t)\|_{L^2({\Bbb R}^3)},\label{ksm2}\\
M_\mu(v(t);c)
&=\sum_{|a|\leq \mu-2}M_2(Z^a v(t);c),\,\,\mu=3,4,
\end{align}
which appear in the following discussion, play an intermediate role. 
We remark that $\partial_t^2$ is absent in the right-hand side of (\ref{ksm2}) above. 
We also use the notation 
\begin{align}
&
\|v\|_{L_r^\infty L_\omega^p({\mathbb R}^3)}
:=
\sup_{r>0}
\|v(r\cdot)\|_{L^p(S^2)},\\
&
\|v\|_{L_r^2 L_\omega^p({\mathbb R}^3)}
:=
\biggl(
\int_0^\infty \|v(r\cdot)\|_{L^p(S^2)}^2 r^2dr
\biggr)^{1/2}.
\end{align}
\begin{lemma}\label{someinequalities2019aug17}
Let $c>0$. 
Suppose that $v$ decays sufficiently fast as $|x|\to\infty$. 
The following inequalities hold for $\alpha=0,1,2,3$
\begin{align}
&
\|
\langle ct-r\rangle\partial_\alpha v(t)
\|_{L^6({\mathbb R}^3)}
\leq
C
\bigl(
N_1(v(t))
+
M_2(v(t);c)
\bigr),\label{ell6}\\
&
\langle ct-r\rangle
|\partial_\alpha v(t,x)|
\leq
C
\biggl(
\sum_{|a|\leq 1}N_1(\partial_x^a v(t))
+
\sum_{|a|\leq 1}M_2(\partial_x^a v(t);c)
\biggr).\label{ellinfty}
\end{align}
Moreover, we have 
\begin{align}
&
\|r\partial_\alpha v(t)\|_{L_r^\infty L_\omega^4({\mathbb R}^3)}
\leq
C
\sum_{|a|+|b|\leq 1}
\|\partial\partial_x^a\Omega^b v(t)\|_{L^2({\mathbb R}^3)},\label{j2}
\\
&
\langle r\rangle
|\partial_\alpha v(t,x)|
\leq
C\sum_{|a|+|b|\leq 2}
\|\partial\partial_x^a\Omega^b v(t)\|_{L^2({\mathbb R}^3)}.\label{j3}
\end{align}
\end{lemma}

Here, we have used the notation 
$\Omega^b:=\Omega_{12}^{b_1}\Omega_{23}^{b_2}\Omega_{13}^{b_3}$ 
for multi-indices $b=(b_1,b_2,b_3)$. 
These inequalities have been already employed in the literature. 
For the proof of (\ref{ell6}), see \cite[(2.10)]{H2016}. 
For the proof of (\ref{ellinfty}), 
see \cite[(37)]{Zha}, \cite[(2.13)]{H2016}. 
See \cite[(3.19)]{Sideris2000} for the proof of (\ref{j2}). 
Finally, combining \cite[(3.14b)]{Sideris2000} 
with the Sobolev embedding $H^2({\mathbb R}^3)\hookrightarrow 
L^\infty({\mathbb R}^3)$, 
we obtain (\ref{j3}).

We also need the following inequality. 
\begin{lemma}
Let $c>0$ and $\alpha=0,1,2,3$. 
Suppose that $v$ decays sufficiently fast as $|x|\to\infty$. 
For any $\theta$ with $0\leq \theta\leq 1/2$, 
there exists a constant $C>0$ such that the inequality 
\begin{equation}\label{interp}
r^{(1/2)+\theta}
\langle ct-r\rangle^{1-\theta}
\|\partial_\alpha v(t,r\cdot)\|_{L^4(S^2)}
\leq
C
\biggl(
\sum_{|a|\leq 1}N_1(\Omega^a v(t))
+
M_2(v(t);c)
\biggr)
\end{equation}
holds.
\end{lemma} 
Following the proof of \cite[(3.19)]{Sideris2000}, 
we are able to obtain this inequality for $\theta=1/2$. 
The next lemma with 
$v=\langle ct-r\rangle\partial_\alpha w$ immediately 
yields (\ref{interp}) for $\theta=0$. 
We follow the idea in Section 2 of \cite{MNS-JJM2005} and 
obtain (\ref{interp}) for $\theta\in (0,1/2)$ by interpolation. 

In our proof, the trace-type inequality also plays an important role. 
For the proof, see, e.g., \cite[(3.16)]{Sideris2000}.
\begin{lemma}\label{traceinequality2019aug17}
There exists a positive constant $C$ such that 
if $v=v(x)$ decays sufficiently fast as $|x|\to\infty$, 
then the inequality
\begin{equation}
r^{1/2}
\|v(r\cdot)\|_{L^4(S^2)}
\leq
C\|\nabla v\|_{L^2({\mathbb R}^3)}
\label{eqn:hoshiro}
\end{equation}
holds.
\end{lemma}
Differently from the analysis in Sideris and Tu \cite{SiderisTu}, 
we need the space-time $L^2$ estimate 
because of the growth of the bound not only in the high energy estimate 
but also in the low energy estimate. 
The following one corresponds to the special case of 
\cite[Theorem 2.1]{HWY2012Adv}.
\begin{lemma}\label{ksstype}
Let $c>0$ and $0<\mu<1/2$. Then, there exists a positive constant $C$ 
depending on $c$ and $\mu$ such that 
the inequality 
\begin{align}\label{l2spacetime}
(1&+T)^{-2\mu}
\left(
\|
r^{-(3/2)+\mu}w
\|^2_{L^2((0,T)\times{\mathbb R}^3)}
+
\|
r^{-(1/2)+\mu}
\partial w
\|^2_{L^2((0,T)\times{\mathbb R}^3)}
\right)
\\
&
\leq
C\|\partial w(0,\cdot)\|^2_{L^2({\mathbb R}^3)}
+C
\int_0^T\!\!\!\int_{{\mathbb R}^3}
\left(
|\partial w||\Box_c w|
+
\frac{|w||\Box_c w|}{r^{1-2\mu}\langle r\rangle^{2\mu}}
\right)dxdt\nonumber
\end{align}
holds for smooth functions $w(t,x)$ 
compactly supported in $x$ for any fixed time.
\end{lemma}
See also Appendix of \cite{Ster} and \cite{MS-SIAM} for earlier and related estimates. 
At first sight, the above estimate may appear useless 
for the proof of global existence, 
because of the presence of the factor $(1+T)^{-2\mu}$. 
Owing to the useful idea of dyadic decomposition 
of the time interval \cite[p.\,363]{Sogge2003} 
(see also (\ref{2019aug171715}) below), 
the estimate (\ref{l2spacetime}) actually works effectively for the proof of 
global existence.  

The following was proved by Klainerman and Sideris. 
\begin{lemma}[Klainerman-Sideris inequality \cite{KS}]\label{lemmaks}
Let $c>0$. There exists a constant $C>0$ such that 
the inequality
\begin{equation}\label{KSineq}
M_2(v(t);c)
\leq
C
\bigl(
N_2(v(t))
+
t
\|\Box_c v(t)\|_{L^2({\mathbb R}^3)}
\bigr)
\end{equation}
holds for smooth functions $v=v(t,x)$ 
decaying sufficiently fast as $|x|\to\infty$. 
\end{lemma}
\section{Bound for $M_\mu(u_1;1)$, $M_\mu(u_2;1)$, 
and $M_\mu(u_3;c_0)$}\label{sectionmmm}
We know that for any data 
$(f_i,g_i)\in C_0^\infty({\mathbb R}^3)\times C_0^\infty({\mathbb R}^3)$ 
$(i=1,2,3)$, 
the Cauchy problem (\ref{eq1})--(\ref{data}) admits 
a unique local (in time) smooth solution which is 
compactly supported in $x$ at any fixed time 
by virtue of finite speed of propagation. 
This section is devoted to the bound for 
$M_\mu(u_1;1)$, $M_\mu(u_2;1)$, and $M_\mu(u_3;c_0)$ 
$(\mu=3,4)$. 
Though much influenced by \cite{SiderisTu}, 
our strategy for establishing their bounds is similar to 
the way adopted in \cite[Section 3]{HZ2019}. 

In the discussion below, 
we use the following quantity for the local solutions 
$u=(u_1,u_2,u_3)$:
\begin{align}\label{<<u>>}
\langle&\!\langle u(t)\rangle\!\rangle\\
&
:=
\langle t\rangle^{-\delta}
\|
r\langle t-r\rangle^{1/2}
\partial u_1(t)
\|_{L^\infty({\mathbb R}^3)}
+
\sum_{|a|\leq 1}
\langle t\rangle^{-2\delta}
\|
r\langle t-r\rangle^{1/2}
\partial Z^a u_1(t)
\|_{L^\infty({\mathbb R}^3)}\nonumber\\
&
+
\|
r\langle t-r\rangle^{1/2}
\partial u_2(t)
\|_{L^\infty({\mathbb R}^3)}
+
\sum_{|a|\leq 1}
\langle t\rangle^{-\delta}
\|
r\langle t-r\rangle^{1/2}
\partial Z^a u_2(t)
\|_{L^\infty({\mathbb R}^3)}\nonumber\\
&
+
\|
r\langle c_0t-r\rangle^{1/2}
\partial u_3(t)
\|_{L^\infty({\mathbb R}^3)}
+
\sum_{|a|\leq 1}
\langle t\rangle^{-\delta}
\|
r\langle c_0t-r\rangle^{1/2}
\partial Z^a u_3(t)
\|_{L^\infty({\mathbb R}^3)}\nonumber\\
&
+
\langle t\rangle^{-\delta}
\sum_{|a|\leq 1}
\biggl(
\|
r^{1/2}
\langle t-r\rangle
\partial Z^a u_1(t)
\|_{L_r^\infty L_\omega^4}
+
\sum_{i=1}^7
\|
r^{1/2}Z_i Z^a u_1(t)
\|_{L_r^\infty L_\omega^4}
\biggr)\nonumber\\
&
+
\langle t\rangle^{-2\delta}
\sum_{|a|\leq 2}
\biggl(
\|
r^{1/2}
\langle t-r\rangle
\partial Z^a u_1(t)
\|_{L_r^\infty L_\omega^4}
+
\sum_{i=1}^7
\|
r^{1/2}Z_i Z^a u_1(t)
\|_{L_r^\infty L_\omega^4}
\biggr)
\nonumber\\
&
+
\sum_{|a|\leq 1}
\biggl(
\|
r^{1/2}
\langle t-r\rangle
\partial Z^a u_2(t)
\|_{L_r^\infty L_\omega^4}
+
\sum_{i=1}^7
\|
r^{1/2}Z_i Z^a u_2(t)
\|_{L_r^\infty L_\omega^4}
\biggr)
\nonumber\\
&
+
\langle t\rangle^{-\delta}
\sum_{|a|\leq 2}
\biggl(
\|
r^{1/2}
\langle t-r\rangle
\partial Z^a u_2(t)
\|_{L_r^\infty L_\omega^4}
+
\sum_{i=1}^7
\|
r^{1/2}Z_i Z^a u_2(t)
\|_{L_r^\infty L_\omega^4}
\biggr)
\nonumber\\
&
+
\sum_{|a|\leq 1}
\biggl(
\|
r^{1/2}
\langle c_0t-r\rangle
\partial Z^a u_3(t)
\|_{L_r^\infty L_\omega^4}
+
\sum_{i=1}^7
\|
r^{1/2}Z_i Z^a u_3(t)
\|_{L_r^\infty L_\omega^4}
\biggr)
\nonumber\\
&
+
\langle t\rangle^{-\delta}
\sum_{|a|\leq 2}
\biggl(
\|
r^{1/2}
\langle c_0t-r\rangle
\partial Z^a u_3(t)
\|_{L_r^\infty L_\omega^4}
+
\sum_{i=1}^7
\|
r^{1/2}Z_i Z^a u_3(t)
\|_{L_r^\infty L_\omega^4}
\biggr)
\nonumber\\
&
+
\sum_{|a|\leq 1}
\bigl(
\langle t\rangle^{-\delta}
\|
r\partial Z^a u_1(t)
\|_{L_r^\infty L_\omega^4}
+
\|
r\partial Z^a u_2(t)
\|_{L_r^\infty L_\omega^4}
+
\|
r\partial Z^a u_3(t)
\|_{L_r^\infty L_\omega^4}
\bigr)\nonumber\\
&
+
\langle t\rangle^{-\delta}
\|
\langle t-r\rangle
\partial u_1(t)
\|_{L^\infty({\mathbb R}^3)}
+
\|
\langle t-r\rangle
\partial u_2(t)
\|_{L^\infty({\mathbb R}^3)}
+
\|
\langle c_0t-r\rangle
\partial u_3(t)
\|_{L^\infty({\mathbb R}^3)}\nonumber\\
&
+
\langle t\rangle^{-\delta}
\sum_{|a|\leq 1}
\bigl(
\langle t\rangle^{-\delta}
\|
\langle t-r\rangle
\partial Z^a u_1(t)
\|_{L^\infty({\mathbb R}^3)}
+
\|
\langle t-r\rangle
\partial Z^a u_2(t)
\|_{L^\infty({\mathbb R}^3)}\nonumber\\
&
\hspace{2.3cm}
+
\|
\langle c_0t-r\rangle
\partial Z^a u_3(t)
\|_{L^\infty({\mathbb R}^3)}
\bigr)\nonumber\\
&
+
\sum_{|a|\leq 1}
\bigl(
\langle t\rangle^{-\delta}
\|
\langle t-r\rangle
\partial Z^a u_1(t)
\|_{L^6({\mathbb R}^3)}
+
\|
\langle t-r\rangle
\partial Z^a u_2(t)
\|_{L^6({\mathbb R}^3)}\nonumber\\
&
\hspace{1.3cm}
+
\|
\langle c_0t-r\rangle
\partial Z^a u_3(t)
\|_{L^6({\mathbb R}^3)}
\bigr)
.\nonumber
\end{align}
Using the constant $\delta$ appearing in Theorem \ref{ourmaintheorem}, 
we also set 
\begin{align}
&
{\mathcal M}_\kappa(u(t))
:=
\langle t\rangle^{-\delta}
M_\kappa(u_1(t);1)
+
M_\kappa(u_2(t);1)
+
M_\kappa(u_3(t);c_0),\\
&
{\mathcal N}_\kappa(u(t))
:=
\langle t\rangle^{-\delta}
N_\kappa(u_1(t))
+
N_\kappa(u_2(t))
+
N_\kappa(u_3(t)).
\end{align}
The purpose of this section is to prove the following:
\begin{proposition}\label{2019june25mnineq} 
Suppose 
\begin{align}
&F_1^{11,\alpha\beta}X_\alpha X_\beta=0,\,
F_2^{11,\alpha\beta}X_\alpha X_\beta
=F_2^{12,\alpha\beta}X_\alpha X_\beta=0,\\
&
\mbox{{\rm and}}\,\,
F_3^{11,\alpha\beta}X_\alpha X_\beta
=F_3^{12,\alpha\beta}X_\alpha X_\beta=0\nonumber
\end{align}
for any $X\in{\mathcal N}^{(1)}$. For $\mu=3,4$, the inequality 
\begin{align}\label{mninequality2019aug16}
{\mathcal M}_\mu(u(t))
\leq&
C_{KS}{\mathcal N}_\mu(u(t))
+
C_{31}
\langle\!\langle u(t)\rangle\!\rangle
{\mathcal N}_\mu(u(t))\\
&
+
C_{32}\langle\!\langle u(t)\rangle\!\rangle^2
{\mathcal N}_3(u(t))
+
C_{33}\langle\!\langle u(t)\rangle\!\rangle
{\mathcal M}_\mu(u(t))\nonumber
\end{align}
holds. 
Here, $C_{KS}$, $C_{31}$, $C_{32}$, and $C_{33}$ 
are positive constants. 
\end{proposition}
The proof of this proposition is carried out 
in the following three subsections.
\subsection{Bound for $M_\mu(u_1;1)$}
We have for $|a|\leq\mu-2$, $\mu=3,4$
\begin{align}\label{equality31}
\Box_1 Z^au_1
=&
\sum\!{}^{'}{\tilde F}_1^{11,\alpha\beta}
(\partial_\alpha Z^{a'}u_1)
(\partial_\beta Z^{a''}u_1)
+
\sum\!{}^{''}{\tilde F}_1^{jk,\alpha\beta}
(\partial_\alpha Z^{a'}u_j)
(\partial_\beta Z^{a''}u_k)\\
&
+
Z^a C_1(\partial u_1,\partial u_2,\partial u_3),\nonumber
\end{align}
where the new coefficients 
${\tilde F}_1^{11,\alpha\beta}$ and 
${\tilde F}_1^{jk,\alpha\beta}$ 
(${\tilde F}_1^{jk,\alpha\beta}=0$ if $j>k$) 
actually depend also on $a'$ and $a''$. 
By $\sum\!{}^{'}$, 
we mean the summation over all 
$a'$ and $a''$ such that 
$|a'|+|a''|\leq |a|$. 
By $\sum\!{}^{''}$, 
we mean the summation over all such $a'$, $a''$ 
and all $j$ and $k$ such that $(j,k)\ne (1,1)$; 
for the second term on the right-hand side above, 
the summation convention 
only over the repeated Greek letters $\alpha$ and $\beta$ 
has been used. 
By Lemma \ref{nullpreserved}, we know
\begin{equation}\label{newnull31}
{\tilde F}_1^{11,\alpha\beta}X_\alpha X_\beta
=0,
\,\,
X\in{\mathcal N}^{(1)}.
\end{equation}
We apply Lemma \ref{lemmaks} to 
$v=Z^a u_1$, $|a|\leq\kappa-2$, $\kappa=3,4$. 
Taking (\ref{KSineq}) into account, 
we need to bound
\begin{align}\label{tj1tj2}
t&\sum\!{}^{'}
\|
{\tilde F}_1^{11,\alpha\beta}
(\partial_\alpha Z^{a'}u_1)
(\partial_\beta Z^{a''}u_1)
\|_{L^2({\mathbb R}^3)}\\
&+
t\sum\!{}^{''}
\|
(\partial Z^{a'}u_j)
(\partial Z^{a''}u_k)
\|_{L^2({\mathbb R}^3)}\nonumber
\end{align}
and 
\begin{equation}\label{2019june25cubic1}
t\sum_{i,j,k}\sum\!{}^{'}
\|\partial u_i(t)\|_{L^\infty({\mathbb R}^3)}
\|
(\partial Z^{a'}u_j)
(\partial Z^{a''}u_k)
\|_{L^2({\mathbb R}^3)}.
\end{equation}
In the following discussion, 
we utilize the characteristic function $\chi_1$ 
of the set 
$\{x\in{\mathbb R}^3:|x|<(c_*/2)t+1\}$, 
where $c_*:=\min\{c_0,1\}$. 
We set $\chi_2:=1-\chi_1$. 
Just for simplicity, 
we omit dependence of $\chi_1$, $\chi_2$ on $t$. 
Owing to (\ref{<<u>>}), we get
\begin{align}\label{2019may261}
\|&
\chi_1
{\tilde F}^{11,\alpha\beta}_1
(\partial_\alpha Z^{a'}u_1)
(\partial_\beta Z^{a''}u_1)
\|_{L^2({\mathbb R}^3)}\\
&
\leq
C
\|
\chi_1
(\partial Z^{a'}u_1)
(\partial Z^{a''}u_1)
\|_{L^2({\mathbb R}^3)}\nonumber\\
&
\leq
C
\langle t\rangle^{-3/2}
\|
r\langle t-r\rangle^{1/2}
\partial Z^{a'}u_1
\|_{L^\infty({\mathbb R}^3)}
\|
r^{-1}\langle t-r\rangle\partial Z^{a''}u_1
\|_{L^2({\mathbb R}^3)}\nonumber\\
&
\leq
C\langle t\rangle^{-(3/2)+2\delta}
\langle\!\langle u(t)\rangle\!\rangle
(
N_{\mu-1}(u_1(t))
+
M_\mu(u_1(t);1)
).\nonumber
\end{align}
Here we have used the Hardy inequality, as in \cite[(6.27)]{H2004}. 
Also, we have assumed $|a'|\leq |a''|$ because the other case can be 
handled similarly. Since $|a'|\leq |a''|\leq |a|\leq\mu-2$ $(\mu=3,4)$, 
we have used the fact $|a'|\leq 1$. 

Since the property (\ref{newnull31}) has played no role above, 
we also obtain by assuming $|a'|\leq |a''|$ 
without loss of generality
\begin{align}\label{2019may262}
\|&
\chi_1
(\partial Z^{a'}u_j)
(\partial Z^{a''}u_k)
\|_{L^2({\mathbb R}^3)}\\
&
\leq
C
\langle t\rangle^{-3/2}
\|
r\langle c_jt-r\rangle^{1/2}
\partial_\alpha Z^{a'}u_j
\|_{L^\infty({\mathbb R}^3)}
\|
r^{-1}\langle c_kt-r\rangle\partial_\beta Z^{a''}u_k
\|_{L^2({\mathbb R}^3)}\nonumber\\
&
\leq
C\langle t\rangle^{-(3/2)+2\delta}
\langle\!\langle u(t)\rangle\!\rangle
\sum_{k=1}^3
(
N_{\mu-1}(u_k(t))
+
M_\mu(u_k(t);c_k)
).\nonumber
\end{align}
Here, and in the following as well, 
by $c_k$ 
we mean $c_1=c_2=1$, $c_3=c_0$ 
(see (\ref{eq1})). 

Let us turn our attention to 
$|x|>(c_*/2)t+1$. 
Using Lemmas \ref{estimationlemma}--\ref{lemma22ofzha} 
together with (\ref{newnull31}), 
we obtain 
\begin{align}\label{2019may263}
\sum&\!{}^{'}
\|
\chi_2
{\tilde F}^{11,\alpha\beta}_1
(\partial_\alpha Z^{a'}u_1)
(\partial_\beta Z^{a''}u_1)
\|_{L^2({\mathbb R}^3)}\\
&
\leq
C\sum_{|a'|+|a''|\leq \mu-2}
\bigl(
\|
\chi_2
|T^{(1)}Z^{a'}u_1|
|\partial Z^{a''}u_1|
\|_{L^2({\mathbb R}^3)}\nonumber\\
&
\hspace{2cm}
+
\|
\chi_2
|\partial Z^{a'}u_1|
|T^{(1)} Z^{a''}u_1|
\|_{L^2({\mathbb R}^3)}
\bigr)\nonumber\\
&
\leq
C\sum_{|a'|+|a''|\leq \mu-2}
\langle t\rangle^{-3/2}
\biggl(
\|
r^{1/2}\partial_t Z^{a'} u_1
\|_{L_r^\infty L_\omega^4}
+
\sum_{i=1}^7
\|
r^{1/2}Z_i Z^{a'} u_1
\|_{L_r^\infty L_\omega^4}\nonumber\\
&
\hspace{2cm}
+
\|
r^{1/2}\langle t-r\rangle\partial_x Z^{a'} u_1
\|_{L_r^\infty L_\omega^4}
\biggr)
\|
\partial Z^{a''} u_1
\|_{L_r^2 L_\omega^4}\nonumber\\
&
\leq
C\langle t\rangle^{-(3/2)+2\delta}
\langle\!\langle u(t)\rangle\!\rangle
N_\mu(u_1(t)).\nonumber 
\end{align}
When dealing with 
$\|\chi_2(\partial Z^{a'}u_j)(\partial Z^{a''}u_k)\|_{L^2}$ 
$(1\leq j\leq k\leq 3,\,(j,k)\ne (1,1))$, 
we obviously know $k=2$ or $k=3$. 
When $|a'|\leq 2$ and $|a''|=0$, 
we get
\begin{align}\label{2019may264}
\|&
\chi_2(\partial Z^{a'}u_j)(\partial u_k)
\|_{L^2({\mathbb R}^3)}\\
&
\leq
C
\langle t\rangle^{-1}
\|\partial Z^{a'}u_j\|_{L^2({\mathbb R}^3)}
\|
r\partial u_k
\|_{L^\infty({\mathbb R}^3)}
\leq
C
\langle t\rangle^{-1+\delta}
\langle\!\langle u(t)\rangle\!\rangle
{\mathcal N}_3(u(t)).\nonumber
\end{align}
When $|a'|\leq 1$ and $|a''|\leq 1$, 
we get
\begin{align}\label{2019may265}
\|&
\chi_2(\partial Z^{a'}u_j)(\partial Z^{a''}u_k)
\|_{L^2({\mathbb R}^3)}\\
&
\leq
C
\langle t\rangle^{-1}
\|r\partial Z^{a'}u_j\|_{L_r^\infty L_\omega^4}
\|
\partial Z^{a''}u_k
\|_{L_r^2 L_\omega^4}
\nonumber\\
&
\leq
C
\langle t\rangle^{-1+\delta}
\langle\!\langle u(t)\rangle\!\rangle
\bigl(
N_3(u_2(t))
+
N_3(u_3(t))
\bigr).\nonumber
\end{align}
When $|a'|=0$ and $|a''|\leq 2$, we get
\begin{align}\label{2019may266}
\|&
\chi_2(\partial u_j)(\partial Z^{a''}u_k)
\|_{L^2({\mathbb R}^3)}\\
&
\leq
C
\langle t\rangle^{-1}
\|
r\partial u_j
\|_{L^\infty({\mathbb R}^3)}
\|
\partial Z^{a''}u_k
\|_{L^2({\mathbb R}^3)}\nonumber\\
&
\leq
C
\langle t\rangle^{-1+\delta}
\langle\!\langle u(t)\rangle\!\rangle
\bigl(
N_3(u_2(t))
+
N_3(u_3(t))
\bigr).\nonumber
\end{align}
As for (\ref{2019june25cubic1}), 
it easy to get for $|a|\leq \mu-2$, $\mu=3,4$
\begin{equation}\label{2019june25ineqcubic1}
t\sum_{i,j,k}\sum\!{}^{'}
\|\partial u_i(t)\|_{L^\infty({\mathbb R}^3)}
\|
(\partial Z^{a'}u_j)
(\partial Z^{a''}u_k)
\|_{L^2({\mathbb R}^3)}
\leq
C
\langle\!\langle u(t)\rangle\!\rangle^2
{\mathcal N}_{\mu-1}(u(t)).
\end{equation}

Summing up, 
we have obtained for $\mu=3,4$
\begin{align}\label{2019june25boundm1}
\langle&t\rangle^{-\delta}
M_\mu(u_1(t);1)\\
&
\leq
C\langle t\rangle^{-\delta}
N_\mu(u_1(t))\nonumber\\
&
+
C\langle t\rangle^{-(1/2)+\delta}
\langle\!\langle u(t)\rangle\!\rangle
\sum_{k=1}^3
\bigl(
N_{\mu-1}(u_k(t))
+
M_\mu(u_k(t);c_k)
\bigr)\nonumber\\
&
+
C\langle t\rangle^{-(1/2)+\delta}
\langle\!\langle u(t)\rangle\!\rangle
N_\mu(u_1(t))
+
C
\bigl(
\langle\!\langle u(t)\rangle\!\rangle
+
\langle\!\langle u(t)\rangle\!\rangle^2
\bigr)
{\mathcal N}_3(u(t)).\nonumber
\end{align}
\subsection{Bound for $M_\mu(u_2;1)$} 
As in (\ref{equality31}), we have
\begin{align}\label{equality32}
\Box_1 Z^a u_2
=&
\sum_{{1\leq j\leq k\leq 3}\atop{(j,k)\ne (1,3)}}
\sum\!{}^{'}{\tilde F}_2^{jk,\alpha\beta}
(\partial_\alpha Z^{a'}u_j)
(\partial_\beta Z^{a''}u_k)\\
&
+Z^a C_2(\partial u_1,\partial u_2,\partial u_3),\nonumber
\end{align}
where the new coefficients ${\tilde F}_2^{jk,\alpha\beta}$ 
actually depend also on $a'$, $a''$. 
By Lemma \ref{nullpreserved}, we know
\begin{equation}\label{newnull32}
{\tilde F}_2^{11,\alpha\beta}X_\alpha X_\beta
=
{\tilde F}_2^{12,\alpha\beta}X_\alpha X_\beta
=
{\tilde F}_2^{22,\alpha\beta}X_\alpha X_\beta
=0,
\,\,
X\in{\mathcal N}^{(1)}.
\end{equation}
(In fact, the condition on 
${\tilde F}_2^{22,\alpha\beta}$ plays no role in the present section.) 
The same computation as in (\ref{2019may261})--(\ref{2019may262}) yields 
\begin{align}\label{june19ineq1}
\|&
\chi_1
{\tilde F}_2^{11,\alpha\beta}
(\partial_\alpha Z^{a'} u_1)
(\partial_\beta Z^{a''}u_1)
\|_{L^2({\mathbb R}^3)}\\
&
\leq
C
\langle t\rangle^{-(3/2)+2\delta}
\langle\!\langle u(t)\rangle\!\rangle
\bigl(
N_{\mu-1}(u_1(t))
+
M_\mu(u_1(t);1)
\bigr),\nonumber
\end{align}
\begin{align}\label{june19ineq2}
\|&
\chi_1
{\tilde F}_2^{12,\alpha\beta}
(\partial_\alpha Z^{a'} u_1)
(\partial_\beta Z^{a''}u_2)
\|_{L^2({\mathbb R}^3)}\\
&
\leq
C
\langle t\rangle^{-(3/2)+2\delta}
\langle\!\langle u(t)\rangle\!\rangle
\bigl(
N_{\mu-1}(u_2(t))
+
M_\mu(u_2(t);1)
\bigr)\nonumber\\
&
+
C
\langle t\rangle^{-(3/2)+\delta}
\langle\!\langle u(t)\rangle\!\rangle
\bigl(
N_{\mu-1}(u_1(t))
+
M_\mu(u_1(t);1)
\bigr).\nonumber
\end{align}
On the other hand, 
using the property (\ref{newnull32}) 
of the coefficients ${\tilde F}_2^{11,\alpha\beta}$ 
and ${\tilde F}_2^{12,\alpha\beta}$, 
we get
\begin{equation}\label{june19ineq3}
\|
\chi_2
{\tilde F}_2^{11,\alpha\beta}
(\partial_\alpha Z^{a'} u_1)
(\partial_\beta Z^{a''}u_1)
\|_{L^2({\mathbb R}^3)}
\leq
C
\langle t\rangle^{-(3/2)+2\delta}
\langle\!\langle u(t)\rangle\!\rangle
N_\mu(u_1(t))
\end{equation}
and 
\begin{align}\label{june19ineq4}
\|&
\chi_2
{\tilde F}_2^{12,\alpha\beta}
(\partial_\alpha Z^{a'} u_1)
(\partial_\beta Z^{a''}u_2)
\|_{L^2({\mathbb R}^3)}\\
&
\leq
C
\langle t\rangle^{-(3/2)+2\delta}
\langle\!\langle u(t)\rangle\!\rangle
N_\mu(u_2(t))
+
C
\langle t\rangle^{-(3/2)+\delta}
\langle\!\langle u(t)\rangle\!\rangle
N_\mu(u_1(t))\nonumber
\end{align}
as in (\ref{2019may263}). 
Therefore, we focus on the terms 
with $(j,k)=(2,2)$, $(2,3)$, and 
$(3,3)$ on the right-hand side of 
(\ref{equality32}). 
We have only to show 
how to estimate the term with $(j,k)=(2,3)$ 
because the others can be handled similarly. 
When $|a'|=0$ and $|a''|\leq 2$, 
we get 
\begin{align}\label{june20930}
\|
\chi_1
(\partial u_2)
(\partial Z^{a''}u_3)
\|_{L^2({\mathbb R}^3)}
&\leq
C
\langle t\rangle^{-1}
\|
\langle t-r\rangle
\partial u_2
\|_{L^\infty({\mathbb R}^3)}
\|
\partial Z^{a''}u_3
\|_{L^2({\mathbb R}^3)}\\
&
\leq
C
\langle t\rangle^{-1}
\langle\!\langle u(t)\rangle\!\rangle
N_3(u_3(t)).\nonumber
\end{align}
When $|a'|\leq 1$ and $|a''|\leq 1$, 
we get
\begin{align}
\|
\chi_1
(\partial Z^{a'}u_2)
(\partial Z^{a''}u_3)
\|_{L^2({\mathbb R}^3)}
&\leq
C
\langle t\rangle^{-1}
\|
\langle t-r\rangle
\partial Z^{a'}u_2
\|_{L^6({\mathbb R}^3)}
\|
\partial Z^{a''}u_3
\|_{L^3({\mathbb R}^3)}\\
&
\leq
C
\langle t\rangle^{-1}
\langle\!\langle u(t)\rangle\!\rangle
N_3(u_3(t)).\nonumber
\end{align}
Furthermore, we obtain 
for $|a'|\leq 2$ and $|a''|=0$
\begin{align}
\|
\chi_1
(\partial Z^{a'}u_2)
(\partial u_3)
\|_{L^2({\mathbb R}^3)}
&\leq
C
\langle t\rangle^{-1}
\|
\partial Z^{a'}u_2
\|_{L^2({\mathbb R}^3)}
\|
\langle c_0t-r\rangle
\partial u_3
\|_{L^\infty({\mathbb R}^3)}\\
&
\leq
C
\langle t\rangle^{-1}
\langle\!\langle u(t)\rangle\!\rangle
N_3(u_2(t)).\nonumber
\end{align}
On the other hand, 
repeating the same discussion as in 
(\ref{2019may264})--(\ref{2019may266}), 
we can obtain 
\begin{equation}\label{june20ineq932}
\|
\chi_2
(\partial Z^{a'}u_2)
(\partial Z^{a''}u_3)
\|_{L^2({\mathbb R}^3)}
\leq
C
\langle t\rangle^{-1}
\langle\!\langle u(t)\rangle\!\rangle
\bigl(
N_3(u_2(t))+N_3(u_3(t))
\bigr)
\end{equation}
for $|a'|+|a''|\leq 2$. 

The cubic term $Z^a C_2(\partial u_1,\partial u_2,\partial u_3)$ 
can be handled in the same way as in (\ref{2019june25ineqcubic1}). 
Summing up, we have obtained for $\mu=3,4$
\begin{align}\label{2019june25boundm2}
M&_\mu(u_2(t);1)\\
&
\leq
CN_\mu(u_2(t))
+
C\langle t\rangle^{-(1/2)+2\delta}
\langle\!\langle u(t)\rangle\!\rangle
\sum_{k=1}^2
\bigl(
N_\mu(u_k(t))
+
M_\mu(u_k(t);1)
\bigr)\nonumber\\
&
+
C
\bigl(
\langle\!\langle u(t)\rangle\!\rangle
+
\langle\!\langle u(t)\rangle\!\rangle^2
\bigr)
{\mathcal N}_3(u(t)).\nonumber
\end{align}
\subsection{Bound for $M_\mu(u_3;c_0)$}
As in (\ref{equality31}), 
we have
\begin{align}\label{equality619}
\Box_{c_0} Z^a u_3
=&
\sum_{{1\leq j\leq k\leq 3}\atop{(j,k)\ne (1,3)}}
\sum\!{}^{'}{\tilde F}_3^{jk,\alpha\beta}
(\partial_\alpha Z^{a'}u_j)
(\partial_\beta Z^{a''}u_k)\\
&
+
Z^a C_3(\partial u_1,\partial u_2,\partial u_3),\nonumber
\end{align}
where the new coefficients above 
actually depend on $a'$, $a''$. 
By Lemma \ref{nullpreserved}, 
we have
\begin{align}
&{\tilde F}_3^{11,\alpha\beta}X_\alpha X_\beta
=
{\tilde F}_3^{12,\alpha\beta}X_\alpha X_\beta=0,
\,\,X\in{\mathcal N}^{(1)},\\
&
{\tilde F}_3^{33,\alpha\beta}X_\alpha X_\beta=0,\,\,
X\in{\mathcal N}^{(c_0)}.
\end{align}
(In fact, this condition on ${\tilde F}_3^{33,\alpha\beta}$ plays no role 
in the present section.) 
The terms with $(j,k)=(1,1)$ and $(1,2)$ on the right-hand side of 
(\ref{equality619}) 
can be handled in the same way as in 
(\ref{june19ineq1}), (\ref{june19ineq3}) and 
(\ref{june19ineq2}), (\ref{june19ineq4}), 
respectively. 
Moreover, we can bound the terms 
with $(j,k)=(2,2)$, $(2,3)$, and $(3,3)$ on the right-hand side of 
(\ref{equality619}) similarly to 
(\ref{june20930})--(\ref{june20ineq932}). 
The cubic term can be handled in the same way as before. 
We have therefore obtained for $\mu=3,4$
\begin{align}\label{2019june25boundm3}
M&_\mu(u_3(t);c_0)\\
&
\leq
CN_\mu(u_3(t))
+
C\langle t\rangle^{-(1/2)+2\delta}
\langle\!\langle u(t)\rangle\!\rangle
\sum_{k=1}^2
\bigl(
N_\mu(u_k(t))
+
M_\mu(u_k(t);1)
\bigr)\nonumber\\
&
+
C
\bigl(
\langle\!\langle u(t)\rangle\!\rangle
+
\langle\!\langle u(t)\rangle\!\rangle^2
\bigr)
{\mathcal N}_3(u(t)).\nonumber
\end{align}
It is obvious that Proposition \ref{2019june25mnineq} 
is a direct consequence of 
(\ref{2019june25boundm1}), 
(\ref{2019june25boundm2}), and 
(\ref{2019june25boundm3}). 
We have finished the proof. $\hfill\Box$
\section{Energy estimate}\label{sectionenergy}
We carry out the energy estimate by relying upon the ghost weight method of Alinhac 
\cite{Al2001}, \cite{Al2010}. 
Just in order to make the proof self-contained, 
let us start our discussion with some preliminaries. 
Let $c>0$, and define 
$m^{\alpha\beta}:={\rm diag}(-1,c^2,c^2,c^2)$. 
We define the energy-momentum tensor as 
\begin{equation}
T^{\alpha\beta}
:=
m^{\alpha\mu}m^{\beta\nu}
(\partial_\mu v)
(\partial_\nu v)
-
\frac12
m^{\alpha\beta}
m^{\mu\nu}
(\partial_\mu v)
(\partial_\nu v).
\end{equation}
A straightforward computation yields
\begin{equation}
\partial_\beta
T^{\alpha\beta}
=
(m^{\alpha\mu}\partial_\mu v)
(-\Box_c v).
\end{equation}
In particular, we have 
\begin{equation}
\partial_\beta T^{0\beta}
=
(\partial_t v)
(\Box_c v).
\end{equation}
For any $g=g(\rho)\in C^1({\mathbb R})$, 
we therefore get
\begin{align}
\partial_\beta
(e^{g(ct-r)}T^{0\beta})
&=
e^{g(ct-r)}g'(ct-r)
(-\omega_\beta)
T^{0\beta}
+
e^{g(ct-r)}\partial_\beta T^{0\beta}\\
&
=
e^{g(ct-r)}
\biggl\{
\frac{c}{2}
g'(ct-r)
\sum_{j=1}^3
(T_j^{(c)}v)^2
+
(\partial_t v)(\Box_c v)
\biggr\}.\nonumber
\end{align}
Here, by $\omega=(\omega_0,\omega_1,\omega_2,\omega_3)$, 
we mean 
$\omega_0=-c$, 
$\omega_j=x_j/|x|$. 
As for $T_j^{(c)}$, see (\ref{2019Nov16OK}). 
With $0<\eta<1/4$, 
we choose 
\begin{equation}
g(\rho)
=
-\int_0^\rho
\langle {\tilde \rho}\rangle^{-1-2\eta}
d{\tilde\rho},\,\,
\rho\in{\mathbb R},
\end{equation}
so that $g'(ct-r)=-\langle ct-r\rangle^{-1-2\eta}$. 
Since $g(\rho)$ is a bounded function and 
we have 
$T^{00}
=
\bigl\{
(\partial_t v)^2
+
c^2|\nabla v|^2
\bigr\}/2$, 
we get the key estimate
\begin{align}\label{keyghostweight}
E&(v(t);c)
+
\sum_{j=1}^3
\int_0^t\!\!
\int_{{\mathbb R}^3}
\langle c\tau-r\rangle^{-1-2\eta}
\bigl(
T_j^{(c)}v(\tau,x)
\bigr)^2
d\tau
dx\\
&
\leq
CE(v(0);c)
+
C
\int_0^t\!\!
\int_{{\mathbb R}^3}
|\Box_c v(\tau,x)|
|\partial_t v(\tau,x)|
d\tau
dx\nonumber
\end{align}
for any smooth function $v(t,x)$ 
decaying sufficiently fast as $|x|\to\infty$. 
In the following, we use the notation for $c>0$
\begin{equation}
G(v(t);c)
:=
\biggl(
\sum_{|a|\leq 3}
\sum_{j=1}^3
\int_{{\mathbb R}^3}
\langle ct-r\rangle^{-1-2\eta}
\bigl(
T_j^{(c)} Z^a v(t,x)
\bigr)^2
dx
\biggr)^{1/2}
\end{equation}
associated with (\ref{keyghostweight}) and 
\begin{equation}\label{weightlocalenergynorm}
L(v(t))
:=
\biggl(
\sum_{|a|\leq 3}
\bigl(
\|
r^{-5/4}Z^a v(t)
\|_{L^2({\mathbb R}^3)}^2
+
\|
r^{-1/4}
\partial Z^a v(t)
\|_{L^2({\mathbb R}^3)}^2
\bigr)
\biggl)^{1/2}
\end{equation}
associated with (\ref{l2spacetime}). 
Recall that we use the notation 
$c_1=c_2=1$, $c_3=c_0$ (see (\ref{eq1})). 
The purpose of this section is to prove 
the following a priori estimate. 
\begin{proposition}\label{2019Nov17OK}
Suppose $c_0\ne 1$ in \mbox{$(\ref{eq1})$} 
and suppose \mbox{$(\ref{assumption1})$--$(\ref{assumption5})$}. 
The unique local $($in time$)$ solution to 
{\rm (\ref{eq1})--(\ref{data})} 
defined in $(0,T)\times{\mathbb R}^3$ for some $T>0$ 
satisfies 
\begin{align}\label{2019Nov17LowEnergy}
\bigl(&
\langle t\rangle^{-\delta}
N_3(u_1(t))
\bigr)^2
+
N_3(u_2(t))^2
+
N_3(u_3(t))^2\\
&
\leq
C\sum_{k=1}^3 N_3(u_k(0))^2\nonumber\\
&
\hspace{0.1cm}
+
C
\langle\!\langle u\rangle\!\rangle_T
\biggl(
\sup_{0<t<T}
\langle t\rangle^{-\delta}
{\mathcal N}_4(u(t))
+
\sup_{0<t<T}
\langle t\rangle^{-\delta}
{\mathcal M}_4(u(t))
\biggr)
\sup_{0<t<T}
{\mathcal N}_3(u(t))\nonumber\\
&
\hspace{0.1cm}
+
C\langle\!\langle u\rangle\!\rangle_T^2
\biggl(
\sup_{0<t<T}
{\mathcal N}_3(u(t))
\biggr)^2\nonumber
\end{align}
and 
\begin{align}\label{2019Nov17HighEnergy}
\bigl(&
\langle t\rangle^{-2\delta}
N_4(u_1(t))
\bigr)^2
+
\sum_{k=2}^3
\bigl(
\langle t\rangle^{-\delta}
N_4(u_k(t))
\bigr)^2\\
&
\hspace{0.1cm}
+
\langle t\rangle^{-4\delta}
\int_0^t
G(u_1(\tau);1)^2d\tau
+
\sum_{k=2}^3
\langle t\rangle^{-2\delta}
\int_0^t
G(u_k(\tau);c_k)^2d\tau
\nonumber\\
&
\leq
C
\sum_{k=1}^3
N_4(u_k(0))^2\nonumber\\
&
\hspace{0.1cm}
+
C
\langle\!\langle u\rangle\!\rangle_T
\int_0^T
\langle\tau\rangle^{-1+2\delta}
\biggl(
\sum_{k=1}^3
L(u_k(\tau))
\biggr)^2d\tau\nonumber\\
&
\hspace{0.1cm}
+
C
\langle\!\langle u\rangle\!\rangle_T
\biggl(
\sup_{0<t<T}
\langle t\rangle^{-\delta}
{\mathcal N}_4(u(t))
\biggr)
\int_0^T
\langle\tau\rangle^{-1+\eta+4\delta}
\sum_{k=1}^3
G(u_k(\tau);c_k)d\tau\nonumber\\
&
\hspace{0.1cm}
+
C
\langle\!\langle u\rangle\!\rangle_T
\biggl(
\sup_{0<t<T}
\langle t\rangle^{-\delta}
{\mathcal N}_4(u(t))
\biggr)^2\nonumber\\
&
\hspace{0.1cm}
+
C
\langle\!\langle u\rangle\!\rangle_T^2
\biggl(
\sup_{0<t<T}
\langle t\rangle^{-\delta}
{\mathcal N}_4(u(t))
+
\sup_{0<t<T}
{\mathcal N}_3(u(t))
\biggr)
\sup_{0<t<T}
\langle t\rangle^{-\delta}
{\mathcal N}_4(u(t))\nonumber
\end{align}
for $0<t<T$. 
$($See $(\ref{definition<<u>>T})$ 
for the definition of 
$\langle\!\langle u\rangle\!\rangle_T$.$)$
\end{proposition}
\subsection{Energy estimate for $u_1$}
Note that (\ref{equality31}) remains valid 
for $|a|\leq 3$. 
Using (\ref{keyghostweight}) and (\ref{equality31}), 
we get for $|a|\leq 3$
\begin{align}\label{u1ghostenergy2019july24}
E&(Z^a u_1(t);1)
+
\sum_{j=1}^3
\int_0^t\!\!\int_{{\mathbb R}^3}
\langle \tau-r\rangle^{-1-2\eta}
\bigl(
T_j^{(1)} Z^a u_1(\tau,x)
\bigr)^2
d\tau dx\\
&
\leq
C
E(Z^a u_1(0);1)
+
C
\sum\!{}^{'}\int_0^t J_{11}\,d\tau
+
C
\sum\!{}^{''}\int_0^t J_{12}\,d\tau
+
C
\int_0^t J_{13}\,d\tau,\nonumber
\end{align}
where 
\begin{align}
&J_{11}
=
\|
{\tilde F}_1^{11,\alpha\beta}
(\partial_\alpha Z^{a'}u_1)
(\partial_\beta Z^{a''}u_1)
(\partial_t Z^a u_1)
\|_{L^1({\mathbb R}^3)},\label{2019july23j11}\\
&
J_{12}
=
\|
{\tilde F}_1^{jk,\alpha\beta}
(\partial_\alpha Z^{a'}u_j)
(\partial_\beta Z^{a''}u_k)
(\partial_t Z^a u_1)
\|_{L^1({\mathbb R}^3)},\label{2019july23j12}\\
&
J_{13}
=
\|
\bigl(
Z^a C_1(\partial u_1,\partial u_2,\partial u_3)
\bigr)
(\partial_t Z^a u_1)
\|_{L^1({\mathbb R}^3)}.\label{2019july23j13}
\end{align}
We refer to (\ref{equality31}) for $\sum\!{}^{'}$ and $\sum\!{}^{''}$. 
As for $|a|\leq 2$ we have only to repeat 
quite the same argument as before. 
Indeed, 
as in (\ref{2019may261}) and (\ref{2019may263}) with $\mu=4$, 
we obtain for $|a|\leq 2$
\begin{equation}
J_{11}
\leq
C
\langle\tau\rangle^{-(3/2)+3\delta}
\langle\!\langle u(\tau)\rangle\!\rangle
\bigl(
N_4(u_1(\tau))
+
M_4(u_1(\tau);1)
\bigr)
\bigl(
\langle\tau\rangle^{-\delta}
N_3(u_1(\tau))
\bigr).
\end{equation}
As in (\ref{2019may262}), (\ref{2019may264})--(\ref{2019may266}), 
we get for $|a|\leq 2$, using the notation $c_1=c_2=1$, $c_3=c_0$
\begin{align}
J_{12}
\leq&
C
\langle\tau\rangle^{-(3/2)+3\delta}
\langle\!\langle u(\tau)\rangle\!\rangle
\biggl(
\sum_{k=1}^3
\bigl(
N_3(u_k(\tau))
+
M_4(u_k(\tau);c_k)
\bigr)
\biggr)
\\
&
\hspace{2cm}
\times
\bigl(
\langle\tau\rangle^{-\delta}
N_3(u_1(\tau))
\bigr)\nonumber\\
&
+
C
\langle\tau\rangle^{-1+2\delta}
\langle\!\langle u(\tau)\rangle\!\rangle
{\mathcal N}_3(u(\tau))
\bigl(
\langle\tau\rangle^{-\delta}
N_3(u_1(\tau))
\bigr).\nonumber
\end{align}
It is also possible to get for $|a|\leq 2$
\begin{equation}\label{j13lowenergy}
J_{13}
\leq
C
\langle\tau\rangle^{-2+4\delta}
\langle\!\langle u(\tau)\rangle\!\rangle^2
{\mathcal N}_3(u(\tau))
\bigl(
\langle\tau\rangle^{-\delta}
N_3(u_1(\tau))
\bigr).
\end{equation}
Therefore, we may focus on $|a|\leq 3$. 
Note that we can no longer rely upon the Hardy inequality 
as we have done in (\ref{2019may261}), (\ref{2019may262}). 
(Its use would cause the loss of derivatives, 
and we could not close the argument.) 
As mentioned in Introduction, 
this is one of the places where we need to proceed quite differently 
from \cite{SiderisTu}, and we utilize the weighted norm 
(\ref{weightlocalenergynorm}) associated with (\ref{l2spacetime}). 
Assuming $|a'|\leq |a''|$ 
(and hence $|a'|\leq 1$) 
without loss of generality, we get 
\begin{align}\label{highenergyj112019july25}
\|&
\chi_1
(\partial Z^{a'} u_1)
(\partial Z^{a''}u_1)
(\partial_t Z^a u_1)
\|_{L^1({\mathbb R}^3)}\\
&
\leq
C
\langle\tau\rangle^{-1}
\|
r^{1/2}
\langle\tau-r\rangle
\partial Z^{a'}u_1
\|_{L^\infty({\mathbb R}^3)}
\|
r^{-1/4}
\partial Z^{a''}u_1
\|_{L^2({\mathbb R}^3)}
\|
r^{-1/4}
\partial_t Z^a u_1
\|_{L^2({\mathbb R}^3)}\nonumber\\
&
\leq
C
\langle\tau\rangle^{-1+2\delta}
\langle\!\langle u(\tau)\rangle\!\rangle
L(u_1(\tau))^2.\nonumber
\end{align}
Here, the Sobolev embedding 
$W^{1,4}(S^2)\hookrightarrow L^\infty(S^2)$ has been used 
to bound 
$\langle\tau\rangle^{-2\delta}
\|
r^{1/2}
\langle\tau-r\rangle
\partial Z^{a'}u_1
\|_{L^\infty({\mathbb R}^3)}$ 
by a constant-multiple of 
$\langle\!\langle u(\tau)\rangle\!\rangle$. 
Similarly, we get for $(j,k)\ne (1,1)$
\begin{align}\label{highenergyj122019july25}
\|&
\chi_1
(\partial Z^{a'} u_j)
(\partial Z^{a''}u_k)
(\partial_t Z^a u_1)
\|_{L^1({\mathbb R}^3)}\\
&
\leq
C
\langle\tau\rangle^{-1+2\delta}
\langle\!\langle u(\tau)\rangle\!\rangle
\biggl(
\sum_{k=1}^3
L(u_k(\tau))
\biggr)
L(u_1(\tau)).\nonumber
\end{align}
On the other hand, 
as in (\ref{2019may263}), 
we employ (\ref{htd20190718}) to get
\begin{align}\label{2019july181530}
\|&
\chi_2
{\tilde F}_1^{11,\alpha\beta}
(\partial_\alpha Z^{a'}u_1)
(\partial_\beta Z^{a''}u_1)
(\partial_t Z^a u_1)
\|_{L^1({\mathbb R}^3)}\\
&
\leq
C
\sum_{|a'|+|a''|\leq 3}
\bigl(
\|
\chi_2
(T^{(1)}Z^{a'}u_1)
(\partial Z^{a''}u_1)
\|_{L^2({\mathbb R}^3)}\nonumber\\
&
\hspace{2.8cm}
+
\|
\chi_2
(\partial Z^{a'}u_1)
(T^{(1)}Z^{a''}u_1)
\|_{L^2({\mathbb R}^3)}
\bigr)N_4(u_1).\nonumber
\end{align}
To continue the estimate of (\ref{2019july181530}), 
we may assume $|a'|\leq |a''|$ 
(hence $|a'|\leq 1$) 
by symmetry. 
Using simply the $L^\infty({\mathbb R}^3)$ norm 
(together with $W^{1,4}(S^2)\hookrightarrow L^\infty(S^2)$) 
and the $L^2$ norm 
in place of the $L_r^\infty L_\omega^4$ and 
the $L_r^2 L_\omega^4$ norms, 
we naturally modify the argument in (\ref{2019may263}) to get
\begin{equation}\label{2019july251729}
\|
\chi_2
(T^{(1)}Z^{a'}u_1)
(\partial Z^{a''}u_1)
\|_{L^2({\mathbb R}^3)}
\leq
C
\langle\tau\rangle^{-(3/2)+2\delta}
\langle\!\langle u(\tau)\rangle\!\rangle
N_4(u_1(\tau)).
\end{equation}
Moreover, using (\ref{interp}) with $\theta=(1/2)-\eta$ and $c=1$, 
we obtain 
\begin{equation}\label{2019july251730}
\|
\chi_2
(\partial Z^{a'}u_1)
(T^{(1)}Z^{a''}u_1)
\|_{L^2({\mathbb R}^3)}
\leq
C
\langle\tau\rangle^{-1+\eta+2\delta}
\langle\!\langle u(\tau)\rangle\!\rangle
G(u_1(\tau);1).
\end{equation}
To handle 
\begin{equation}
\sum\!{}^{''}
\|
\chi_2
{\tilde F}^{jk,\alpha\beta}_1
(\partial_\alpha Z^{a'} u_j)
(\partial_\beta Z^{a''} u_k)
(\partial_t Z^a u_1)
\|_{L^1({\mathbb R}^3)},
\end{equation}
we focus on the estimate of 
\begin{equation}\label{2019july191626}
\|
\chi_2
(\partial Z^{a'} u_j)
(\partial Z^{a''} u_k)
(\partial_t Z^a u_1)
\|_{L^1({\mathbb R}^3)}
\end{equation}
for $|a|\leq 3$, 
$|a'|+|a''|\leq 3$, and 
$(j,k)\ne (1,1)$, 
because of lack of the null condition 
on the coefficients 
$\{F_1^{jk,\alpha\beta}\}$ with 
$(j,k)\ne (1,1)$. 
Unlike (\ref{2019july181530}), 
we fully utilize the different growth rates for 
the high energy and the low energy of $u_1$. 
Without loss of generality, 
we may suppose $j\ne 1$ in (\ref{2019july191626}). 
When $|a'|=0$ (and hence $|a''|\leq 3$), 
we get
\begin{align}\label{2019july191641}
\|&
\chi_2
(\partial u_j)
(\partial Z^{a''} u_k)
(\partial_t Z^a u_1)
\|_{L^1({\mathbb R}^3)}\\
&
\leq
C
\langle\tau\rangle^{-1+4\delta}
\|
r\partial u_j
\|_{L^\infty({\mathbb R}^3)}
\bigl(
\langle\tau\rangle^{-2\delta}
\|
\partial Z^{a''} u_k
\|_{L^2({\mathbb R}^3)}
\bigr)
\bigl(
\langle\tau\rangle^{-2\delta}
N_4(u_1(\tau))
\bigr)\nonumber\\
&
\leq
C
\langle\tau\rangle^{-1+4\delta}
\langle\!\langle u(\tau)\rangle\!\rangle
\bigl(
\langle\tau\rangle^{-\delta}
{\mathcal N}_4(u(\tau))
\bigr)
\bigl(
\langle\tau\rangle^{-2\delta}
N_4(u_1(\tau))
\bigr).\nonumber
\end{align}
When $|a'|=1$ (and hence $|a''|\leq 2$), 
we employ the $L_r^\infty L_\omega^4$ norm 
and the $L_r^2 L_\omega^4$ norm 
(together with $W^{1,2}(S^2)\hookrightarrow L^4(S^2)$) 
in place of the $L^\infty({\mathbb R}^3)$ norm and 
the $L^2({\mathbb R}^3)$ norm, 
to get the same bound as in (\ref{2019july191641}). 
When $|a'|=2$ (and hence $|a''|\leq 1$), 
we obtain
\begin{align}\label{2019july191656}
\|&
\chi_2
(\partial Z^{a'}u_j)
(\partial Z^{a''} u_k)
(\partial_t Z^a u_1)
\|_{L^1({\mathbb R}^3)}\\
&
\leq
C
\langle\tau\rangle^{-1+4\delta}
\bigl(
\langle\tau\rangle^{-\delta}
\|
\partial Z^{a'} u_j
\|_{L_r^2 L_\omega^4}
\bigr)
\bigl(
\langle\tau\rangle^{-\delta}
\|
r\partial Z^{a''}u_k
\|_{L_r^\infty L_\omega^4}
\bigr)
\bigl(
\langle\tau\rangle^{-2\delta}
N_4(u_1(\tau))
\bigr)\nonumber\\
&
\leq
C
\langle\tau\rangle^{-1+4\delta}
\bigl(
\langle\tau\rangle^{-\delta}
{\mathcal N}_4(u(\tau))
\bigr)
\langle\!\langle u(\tau)\rangle\!\rangle
\bigl(
\langle\tau\rangle^{-2\delta}
N_4(u_1(\tau))
\bigr).\nonumber
\end{align}
For $|a'|=3$ (and hence $|a''|=0$), 
we employ the $L^2({\mathbb R}^3)$ norm 
and the $L^\infty({\mathbb R}^3)$ norm 
in place of the $L_r^2 L_\omega^4$ norm 
and the $L_r^\infty L_\omega^4$ norm, 
to get the same bound as in (\ref{2019july191656}). 

It remains to bound (\ref{2019july23j13}) for $|a|\leq 3$. 
It is possible to get 
\begin{equation}\label{2019july231855}
J_{13}
\leq
C
\langle\tau\rangle^{-2+6\delta}
\langle\!\langle u(\tau)\rangle\!\rangle^2
\bigl(
\langle\tau\rangle^{-\delta}
{\mathcal N}_4(u(\tau))
+
{\mathcal N}_3(u(\tau))
\bigr)
\bigl(
\langle\tau\rangle^{-2\delta}
N_4(u_1(\tau))
\bigr).
\end{equation}
It suffices to handle such a typical cubic term as 
$
(\partial_t Z^{a'}u_1)
(\partial_t Z^{a''}u_1)
(\partial_t Z^{a'''}u_1)
$ 
with 
$|a'|+|a''|+|a'''|=3$, 
to show (\ref{2019july231855}). 
We get 
\begin{align}
\biggl(
&
\sum_{|a'|=3}
\|
\chi_1
(\partial_t Z^{a'}u_1)
(\partial_t u_1)^2
\|_{L^2({\mathbb R}^3)}\\
&
\hspace{0.1cm}
+
\sum_{{|a'|=2}\atop{|a''|=1}}
\|
\chi_1
(\partial_t Z^{a'}u_1)
(\partial_t Z^{a''}u_1)
(\partial_t u_1)
\|_{L^2({\mathbb R}^3)}\nonumber\\
&
\hspace{0.1cm}
+
\sum_{{|a'|=|a'|}\atop{=|a'''|=1}}
\|
\chi_1
(\partial_t Z^{a'}u_1)
(\partial_t Z^{a''}u_1)
(\partial_t Z^{a'''}u_1)
\|_{L^2({\mathbb R}^3)}
\biggr)
N_4(u_1)\nonumber\\
&
\leq
C
\langle\tau\rangle^{-2}
\biggl(
\sum_{|a'|=3}
\|
\partial_t Z^{a'}u_1
\|_{L^2({\mathbb R}^3)}
\|\langle\tau-r\rangle
\partial_t u_1
\|_{L^\infty({\mathbb R}^3)}^2
\nonumber\\
&
\hspace{1.4cm}
+
\sum_{{|a'|=2}\atop{|a''|=1}}
\|
\partial_t Z^{a'}u_1
\|_{L^3({\mathbb R}^3)}
\|
\langle\tau-r\rangle
\partial_t Z^{a''}u_1
\|_{L^6({\mathbb R}^3)}
\|
\langle\tau-r\rangle
\partial_t u_1
\|_{L^\infty({\mathbb R}^3)}\nonumber\\
&
\hspace{1.4cm}
+
\sum_{{|a'|=|a'|}\atop{=|a'''|=1}}
\|
\langle\tau-r\rangle
\partial_t Z^{a'}u_1
\|_{L^\infty({\mathbb R}^3)}
\|
\langle\tau-r\rangle
\partial_t Z^{a''}u_1
\|_{L^6({\mathbb R}^3)}\nonumber\\
&
\hspace{7cm}
\times
\|
\partial_t Z^{a'''}u_1
\|_{L^3({\mathbb R}^3)}
\biggr)
N_4(u_1)\nonumber\\
&
\leq
C
\langle\tau\rangle^{-2+6\delta}
\langle\!\langle u(\tau)\rangle\!\rangle^2
\bigl(
\langle\tau\rangle^{-2\delta}
N_4(u_1(\tau))
+
\langle\tau\rangle^{-\delta}
N_3(u_1(\tau))
\bigr)\nonumber\\
&
\hspace{3.4cm}
\times
\bigl(
\langle\tau\rangle^{-2\delta}
N_4(u_1(\tau))
\bigr).\nonumber
\end{align}
We also obtain
\begin{align}
\biggl(
&
\sum_{|a'|=3}
\|
\chi_2
(\partial_t Z^{a'}u_1)
(\partial_t u_1)^2
\|_{L^2({\mathbb R}^3)}\\
&
\hspace{0.1cm}
+
\sum_{{|a'|=2}\atop{|a''|=1}}
\|
\chi_2
(\partial_t Z^{a'}u_1)
(\partial_t Z^{a''}u_1)
(\partial_t u_1)
\|_{L^2({\mathbb R}^3)}\nonumber\\
&
\hspace{0.1cm}
+
\sum_{{|a'|=|a'|}\atop{=|a'''|=1}}
\|
\chi_2
(\partial_t Z^{a'}u_1)
(\partial_t Z^{a''}u_1)
(\partial_t Z^{a'''}u_1)
\|_{L^2({\mathbb R}^3)}
\biggr)
N_4(u_1)\nonumber\\
&
\leq
C
\langle\tau\rangle^{-2}
\biggl(
\sum_{|a'|=3}
\|
\partial_t Z^{a'}u_1
\|_{L^2({\mathbb R}^3)}
\|
r\partial_t u_1
\|_{L^\infty({\mathbb R}^3)}^2\nonumber\\
&
\hspace{0.1cm}
+
\sum_{{|a'|=2}\atop{|a''|=1}}
\|
\partial_t Z^{a'}u_1
\|_{L^2_r L^4_\omega}
\|
r\partial_t Z^{a''}u_1
\|_{L^\infty_r L^4_\omega}
\|
r\partial_t u_1
\|_{L^\infty({\mathbb R}^3)}\nonumber\\
&
\hspace{0.1cm}
+
\sum_{{|a'|=|a'|}\atop{=|a'''|=1}}
\|
\partial_t Z^{a'}u_1
\|_{L^2_r L^\infty_\omega}
\|
r\partial_t Z^{a''}u_1
\|_{L^\infty_r L^4_\omega}
\|
r\partial_t Z^{a'''}u_1
\|_{L^\infty_r L^4_\omega}
\biggr)
N_4(u_1)\nonumber\\
&
\leq
C
\langle\tau\rangle^{-2+6\delta}
\bigl(
\langle\tau\rangle^{-2\delta}
N_4(u_1)
\bigr)^2
\langle\!\langle u(\tau)\rangle\!\rangle^2.\nonumber
\end{align}
With the notation 
\begin{equation}\label{definition<<u>>T}
\langle\!\langle u\rangle\!\rangle_T
:=
\sup_{0<t<T}
\langle\!\langle u(t)\rangle\!\rangle,
\end{equation}
summing yields for $|a|\leq 2$
\begin{align}\label{2019aug201631}
\langle&t\rangle^{-2\delta}E(Z^a u_1(t);1)
\\
&
\leq
C
E(Z^a u_1(0);1)\nonumber\\
&
+
C
\langle\!\langle u\rangle\!\rangle_T
\biggl(
\sup_{0<t<T}
\langle t\rangle^{-\delta}
{\mathcal N}_4(u(t))
+
\sup_{0<t<T}
\langle t\rangle^{-\delta}
{\mathcal M}_4(u(t))
\biggr)
\sup_{0<t<T}
{\mathcal N}_3(u(t))\nonumber\\
&
+
C
\langle\!\langle u\rangle\!\rangle_T
\biggl(
\sup_{0<t<T}
{\mathcal N}_3(u(t))
\biggr)^2
+
C
\langle\!\langle u\rangle\!\rangle_T^2
\biggl(
\sup_{0<t<T}
{\mathcal N}_3(u(t))
\biggr)^2,\nonumber
\end{align}
and for $|a|\leq 3$
\begin{align}\label{2019aug61027}
\langle&t\rangle^{-4\delta}
E(Z^a u_1(t);1)
+
\langle t\rangle^{-4\delta}
\int_0^t
G(u_1(\tau);1)^2
d\tau\\
&
\leq
C
E(Z^a u_1(0);1)\nonumber\\
&
+
C
\langle\!\langle u\rangle\!\rangle_T
\int_0^t
\langle\tau\rangle^{-1+2\delta}
\biggl(
\sum_{k=1}^3
L(u_k(\tau))
\biggr)
L(u_1(\tau))
d\tau\nonumber\\
&
+
C
\langle\!\langle u\rangle\!\rangle_T
\biggl(
\sup_{0<t<T}
\langle t\rangle^{-\delta}
{\mathcal N}_4(u(t))
\biggr)
\int_0^t
\langle\tau\rangle^{-1+\eta+4\delta}
G(u_1(\tau);1)
d\tau\nonumber\\
&
+
C
\langle\!\langle u\rangle\!\rangle_T
\biggl(
\sup_{0<t<T}
\langle t\rangle^{-\delta}
{\mathcal N}_4(u(t))
\biggr)^2\nonumber\\
&
+
C
\langle\!\langle u\rangle\!\rangle_T^2
\biggl(
\sup_{0<t<T}
\langle t\rangle^{-\delta}
{\mathcal N}_4(u(t))
+
\sup_{0<t<T}
{\mathcal N}_3(u(t))
\biggr)
\sup_{0<t<T}
\langle t\rangle^{-\delta}
{\mathcal N}_4(u(t)).\nonumber
\end{align}
\subsection{Energy estimate for $u_2$.} 
As in (\ref{u1ghostenergy2019july24}), we get for $|a|\leq 3$
\begin{align}\label{u2ghostenergy2019july24}
E&(Z^a u_2(t);1)
+
\sum_{j=1}^3
\int_0^t\!\!\int_{{\mathbb R}^3}
\langle \tau-r\rangle^{-1-2\eta}
\bigl(
T_j^{(1)} Z^a u_2(\tau,x)
\bigr)^2
d\tau dx\\
&
\leq
C
E(Z^a u_2(0);1)\nonumber\\
&
\hspace{0.1cm}
+
C
\sum_{{(j,k)=(1,1),}\atop{(1,2),(2,2)}}
\sum\!{}^{'}
\int_0^t J_{21}\,d\tau
+
C
\sum_{{(j,k)=(2,3),}\atop{(3,3)}}
\sum\!{}^{'}
\int_0^t J_{21}\,d\tau
+
C
\int_0^t J_{22}\,d\tau,\nonumber
\end{align}
here we have set
\begin{equation}
J_{21}
=J_{21}^{(j,k)}
:=
\|
{\tilde F}_2^{jk,\alpha\beta}
(\partial_\alpha Z^{a'}u_j)
(\partial_\beta Z^{a''}u_k)
(\partial_t Z^a u_2)
\|_{L^1({\mathbb R}^3)}
\end{equation}
(Note that the summation convention 
only for the Greek letters $\alpha$ and $\beta$ 
has been used above, and 
the coefficients ${\tilde F}_2^{jk,\alpha\beta}$ actually 
depend also on $a'$, $a''$.), 
and 
\begin{equation}
J_{22}
:=
\|
\bigl(
Z^a C_2(\partial u_1,\partial u_2,\partial u_3)
\bigr)
(\partial_t Z^a u_2)
\|_{L^1({\mathbb R}^3)}.
\end{equation}
Let us first consider the low energy $|a|\leq 2$. 
As in (\ref{june19ineq1})--(\ref{june19ineq2}), 
it is possible to obtain 
\begin{align}\label{2019july281550}
\|&
\chi_1
(\partial Z^{a'}u_j)
(\partial Z^{a''}u_k)
(\partial_t Z^a u_2)
\|_{L^1({\mathbb R}^3)}\\
&
\leq
C
\langle\tau\rangle^{-(3/2)+4\delta}
\langle\!\langle u(\tau)\rangle\!\rangle
\bigl(
{\mathcal N}_3(u(\tau))
+
\langle\tau\rangle^{-\delta}
{\mathcal M}_4(u(\tau))
\bigr)
N_3(u_2(\tau)).\nonumber
\end{align}
On the other hand, 
for $(j,k)=(1,1), (1,2)$, and $(2,2)$, 
we benefit from the null condition and obtain 
\begin{align}\label{2019july281551}
\|&
\chi_2
{\tilde F}_2^{jk,\alpha\beta}
(\partial_\alpha Z^{a'}u_j)
(\partial_\beta Z^{a''}u_k)
(\partial_t Z^a u_2)
\|_{L^1({\mathbb R}^3)}\\
&
\leq
C
\langle\tau\rangle^{-(3/2)+4\delta}
\langle\!\langle u(\tau)\rangle\!\rangle
\bigl(
\langle\tau\rangle^{-2\delta}
N_4(u_1(\tau))
+
\langle\tau\rangle^{-\delta}
N_4(u_2(\tau))
\bigr)
N_3(u_2(\tau))\nonumber
\end{align}
as in (\ref{2019may263}). 
For $(j,k)=(2,3), (3,3)$, 
we divide the set 
$\{x\in{\mathbb R}^3\,:\,|x|>(c_*/2)t+1\}$ 
($c_*=\min\{c_0,1\}$) into 
$$
\biggl\{
x\in{\mathbb R}^3\,:\,
\frac{c_*}{2}t+1<|x|<\frac{c_0+1}{2}t+1
\biggr\}
\mbox{ and }
\biggl\{
x\in{\mathbb R}^3\,:\,
|x|>\frac{c_0+1}{2}t+1
\biggr\},
$$
and obtain 
for $j=2,3$, $|a'|+|a''|\leq 2$, and $|a|\leq 2$
\begin{align}\label{2019july281626}
\|&
\chi_2
(\partial Z^{a'}u_j)
(\partial Z^{a''}u_3)
(\partial_t Z^a u_2)
\|_{L^1({\mathbb R}^3)}\\
&
\leq
C
\langle\tau\rangle^{-(3/2)}
\|
\partial Z^{a'}u_j
\|_{L^2({\mathbb R}^3)}
\bigl(
\|
r^{1/2}
\langle c_0\tau-r\rangle
\partial Z^{a''}u_3
\|_{L_r^\infty L_\omega^4}
\|
\partial_t Z^a u_2
\|_{L_r^2 L_\omega^4}\nonumber\\
&
\hspace{5cm}
+
\|
\partial Z^{a''}u_3
\|_{L_r^2 L_\omega^4}
\|
r^{1/2}
\langle\tau-r\rangle
\partial_t Z^a u_2
\|_{L_r^\infty L_\omega^4}
\bigr)\nonumber\\
&
\leq
C
\langle\tau\rangle^{-(3/2)+2\delta}
\langle\!\langle u(\tau)\rangle\!\rangle
\bigl(
N_3(u_2(\tau))+N_3(u_3(\tau))
\bigr)\nonumber\\
&
\hspace{3.7cm}
\times
\bigl(
\langle\tau\rangle^{-\delta}
N_4(u_2(\tau))
+
\langle\tau\rangle^{-\delta}
N_4(u_3(\tau))
\bigr)\nonumber
\end{align}
by considering the two cases
$c_0<1$ and $c_0>1$, separately. 
It is also possible to get for $|a|\leq 2$
\begin{equation}\label{j22lowenergy}
J_{22}
\leq
C
\langle\tau\rangle^{-2+3\delta}
\langle\!\langle u(\tau)\rangle\!\rangle^2
{\mathcal N}_3(u(\tau))
N_3(u_2(\tau)).
\end{equation}
Summing yields for $|a|\leq 2$
\begin{align}\label{2019aug201638}
E&(Z^a u_2(t);1)
\\
&
\leq
C
E(Z^a u_2(0);1)\nonumber\\
&
\hspace{0.1cm}
+
C
\langle\!\langle u\rangle\!\rangle_T
\biggl(
\sup_{0<t<T}
\langle t\rangle^{-\delta}
{\mathcal N}_4(u(t))
+
\sup_{0<t<T}
\langle t\rangle^{-\delta}
{\mathcal M}_4(u(t))
\biggr)
\sup_{0<t<T}
{\mathcal N}_3(u(t))\nonumber\\
&
\hspace{0.1cm}
+
C
\langle\!\langle u\rangle\!\rangle_T^2
\biggl(
\sup_{0<t<T}
{\mathcal N}_3(u(t))
\biggr)^2.\nonumber
\end{align}
Let us turn our attention to the high energy $|a|\leq 3$. 
Proceeding as in (\ref{highenergyj112019july25}) 
and (\ref{highenergyj122019july25}), 
we get for $|a'|+|a''|\leq 3$
\begin{align}\label{2019july291431}
\|&
\chi_1
(\partial Z^{a'}u_j)
(\partial Z^{a''}u_k)
(\partial_t Z^a u_2)
\|_{L^1({\mathbb R}^3)}\\
&
\leq
C
\langle\tau\rangle^{-1+2\delta}
\langle\!\langle u(\tau)\rangle\!\rangle
\biggl(
\sum_{k=1}^3
L(u_k(\tau))
\biggr)
L(u_2(\tau)).\nonumber
\end{align}
On the other hand, 
for $(j,k)=(1,1), (1,2), (2,2)$, 
we rely upon the null condition to get
\begin{align}\label{2019july291507}
&
\sum_{{(j,k)=(1,1),}\atop{(1,2),(2,2)}}
\|
\chi_2
{\tilde F}_2^{jk,\alpha\beta}
(\partial Z^{a'}u_j)
(\partial Z^{a''}u_k)
(\partial_t Z^a u_2)
\|_{L^1({\mathbb R}^3)}\\
&
\leq
C
\langle\tau\rangle^{-(3/2)+4\delta}
\langle\!\langle u(\tau)\rangle\!\rangle
\bigl(
\langle\tau\rangle^{-2\delta}
N_4(u_1(\tau))
+
\langle\tau\rangle^{-\delta}
N_4(u_2(\tau))
\bigr)
N_4(u_2(\tau))\nonumber\\
&
\hspace{0.1cm}
+
C
\langle\tau\rangle^{-1+\eta+2\delta}
\langle\!\langle u(\tau)\rangle\!\rangle
\biggl(
\sum_{i=1,2}
G(u_i(\tau);1)
\biggr)
N_4(u_2(\tau))\nonumber
\end{align}
in the same way as in (\ref{2019july181530}), 
(\ref{2019july251729}), and (\ref{2019july251730}). 
For $(j,k)=(2,3), (3,3)$, 
we can no longer rely upon the null condition. 
Instead, we rely upon the fact 
$\min\{|a'|,|a''|\}\leq 1$ 
for $|a'|+|a''|\leq 3$. 
Proceeding as in (\ref{2019july191641}) and (\ref{2019july191656}), 
we then obtain
\begin{align}\label{2019july291525}
&
\sum_{j=2,3}
\|
\chi_2
(\partial Z^{a'}u_j)
(\partial Z^{a''}u_3)
(\partial_t Z^a u_2)
\|_{L^1({\mathbb R}^3)}\\
&
\hspace{0.1cm}
\leq
C
\langle\tau\rangle^{-1+2\delta}
\langle\!\langle u(\tau)\rangle\!\rangle
\bigl(
\langle\tau\rangle^{-\delta}
N_4(u_2(\tau))
+
\langle\tau\rangle^{-\delta}
N_4(u_3(\tau))
\bigr)
\langle\tau\rangle^{-\delta}
N_4(u_2(\tau)).\nonumber
\end{align}
Finally, we get for $|a|\leq 3$
\begin{equation}
J_{22}
\leq
C
\langle\tau\rangle^{-2+5\delta}
\langle\!\langle u(\tau)\rangle\!\rangle^2
\bigl(
\langle\tau\rangle^{-\delta}
{\mathcal N}_4(u(\tau))
+
{\mathcal N}_3(u(\tau))
\bigr)
\bigl(
\langle\tau\rangle^{-\delta}
N_4(u_2(\tau))
\bigr)
\end{equation}
in the same way as in (\ref{2019july231855}). 
Summing yields for $|a|\leq 3$
\begin{align}\label{2019aug61028}
\langle&t\rangle^{-2\delta}
E(Z^a u_2(t);1)
+
\langle t\rangle^{-2\delta}
\int_0^t
G(u_2(\tau);1)^2
d\tau\\
&
\leq
C
E(Z^a u_2(0);1)\nonumber\\
&
\hspace{0.1cm}
+
C
\langle\!\langle u\rangle\!\rangle_T
\int_0^t
\langle\tau\rangle^{-1+2\delta}
\biggl(
\sum_{k=1}^3
L(u_k(\tau))
\biggr)
L(u_2(\tau))d\tau\nonumber\\
&
\hspace{0.1cm}
+
C
\langle\!\langle u\rangle\!\rangle_T
\biggl(
\sup_{0<t<T}
\langle t\rangle^{-\delta}
{\mathcal N}_4(u(t))
\biggr)
\int_0^t
\langle\tau\rangle^{-1+\eta+3\delta}
\biggl(
\sum_{i=1,2}
G(u_i(\tau);1)
\biggr)
d\tau\nonumber\\
&
\hspace{0.1cm}
+
C
\langle\!\langle u\rangle\!\rangle_T
\biggl(
\sup_{0<t<T}
\langle t\rangle^{-\delta}
{\mathcal N}_4(u(t))
\biggr)^2\nonumber\\
&
\hspace{0.1cm}
+
C
\langle\!\langle u\rangle\!\rangle_T^2
\biggl(
\sup_{0<t<T}
\langle t\rangle^{-\delta}
{\mathcal N}_4(u(t))
+
\sup_{0<t<T}
{\mathcal N}_3(u(t))
\biggr)
\sup_{0<t<T}
\langle t\rangle^{-\delta}
{\mathcal N}_4(u(t)).\nonumber
\end{align}
\subsection{Energy estimate for $u_3$.} 
As in (\ref{u1ghostenergy2019july24}), we get for $|a|\leq 3$
\begin{align}\label{u3ghostenergy2019july24}
E&(Z^a u_3(t);c_0)
+
\sum_{j=1}^3
\int_0^t\!\!\int_{{\mathbb R}^3}
\langle c_0\tau-r\rangle^{-1-2\eta}
\bigl(
T_j^{(c_0)} Z^a u_3(\tau,x)
\bigr)^2
d\tau dx\\
&
\leq
C
E(Z^a u_3(0);c_0)
+
C
\sum_{{(j,k)=(1,1),}\atop{(1,2)}}
\sum\!{}^{'}
\int_0^t J_{31}\,d\tau
+
C
\sum\!{}^{'}
\int_0^t J_{32}\,d\tau\nonumber\\
&
+
C
\sum_{k=2,3}
\sum\!{}^{'}
\int_0^t J_{33}\,d\tau
+
C
\int_0^t J_{34}\,d\tau.\nonumber
\end{align}
Here we have set
\begin{equation}
J_{31}
=J_{31}^{(j,k)}
:=
\|
{\tilde F}_3^{jk,\alpha\beta}
(\partial_\alpha Z^{a'}u_j)
(\partial_\beta Z^{a''}u_k)
(\partial_t Z^a u_3)
\|_{L^1({\mathbb R}^3)},
\end{equation}
(Note that the summation convention 
only for the Greek letters $\alpha$ and $\beta$ 
has been used above.)
\begin{align}
&J_{32}
:=
\|
{\tilde F}_3^{33,\alpha\beta}
(\partial_\alpha Z^{a'}u_3)
(\partial_\beta Z^{a''}u_3)
(\partial_t Z^a u_3)
\|_{L^1({\mathbb R}^3)},\\
&
J_{33}=J_{33}^{(k)}
:=
\|
{\tilde F}_3^{2k,\alpha\beta}
(\partial_\alpha Z^{a'}u_2)
(\partial_\beta Z^{a''}u_k)
(\partial_t Z^a u_3)
\|_{L^1({\mathbb R}^3)},
\end{align}
(Note that 
the coefficients ${\tilde F}_3^{jk,\alpha\beta}$ actually 
depend also on $a'$, $a''$.), 
and 
\begin{equation}
J_{34}
:=
\|
\bigl(
Z^a C_3(\partial u_1,\partial u_2,\partial u_3)
\bigr)
(\partial_t Z^a u_3)
\|_{L^1({\mathbb R}^3)}.
\end{equation}
Let us first consider the low energy $|a|\leq 2$. 
In the same way as in (\ref{2019july281550})--(\ref{2019july281551}), 
we obtain
\begin{equation}
J_{31}
\leq
C
\langle\tau\rangle^{-(3/2)+4\delta}
\langle\!\langle u(\tau)\rangle\!\rangle
\bigl(
\langle\tau\rangle^{-\delta}
{\mathcal N}_4(u(\tau))
+
\langle\tau\rangle^{-\delta}
{\mathcal M}_4(u(\tau))
\bigr)
N_3(u_3(\tau)).
\end{equation}
Since $\{{\tilde F}_3^{33,\alpha\beta}\}$ satisfies the null condition 
(\ref{assumption3}), 
we also get
\begin{equation}
J_{32}
\leq
C
\langle\tau\rangle^{-(3/2)+2\delta}
\langle\!\langle u(\tau)\rangle\!\rangle
\bigl(
\langle\tau\rangle^{-\delta}
{\mathcal N}_4(u(\tau))
+
\langle\tau\rangle^{-\delta}
{\mathcal M}_4(u(\tau))
\bigr)
N_3(u_3(\tau)).
\end{equation}
For $J_{33}$, we proceed as in (\ref{2019july281550}) 
and (\ref{2019july281626}), to get
\begin{equation}
J_{33}
\leq
C
\langle\tau\rangle^{-(3/2)+2\delta}
\langle\!\langle u(\tau)\rangle\!\rangle
\bigl(
\langle\tau\rangle^{-\delta}
{\mathcal N}_4(u(\tau))
+
\langle\tau\rangle^{-\delta}
{\mathcal M}_4(u(\tau))
\bigr)
{\mathcal N}_3(u(\tau)).
\end{equation}
It is possible to get for $|a|\leq 2$
\begin{equation}
J_{34}
\leq
C
\langle\tau\rangle^{-2+3\delta}
\langle\!\langle u(\tau)\rangle\!\rangle^2
{\mathcal N}_3(u(\tau))
N_3(u_3(\tau)).
\end{equation}
Summing yields for $|a|\leq 2$
\begin{align}\label{2019aug201641}
E&(Z^a u_3(t);c_0)
\\
&
\leq
C
E(Z^a u_3(0);c_0)\nonumber\\
&
\hspace{0.1cm}
+
C
\langle\!\langle u\rangle\!\rangle_T
\biggl(
\sup_{0<t<T}
\langle t\rangle^{-\delta}
{\mathcal N}_4(u(t))
+
\sup_{0<t<T}
\langle t\rangle^{-\delta}
{\mathcal M}_4(u(t))
\biggr)
\sup_{0<t<T}
{\mathcal N}_3(u(t))\nonumber\\
&
\hspace{0.1cm}
+
C
\langle\!\langle u\rangle\!\rangle_T^2
\biggl(
\sup_{0<t<T}{\mathcal N}_3(u(t))
\biggr)^2.\nonumber
\end{align}
As for the high energy $|a|\leq 3$, 
we obtain
\begin{align}
J&_{31},\,J_{32}\\
&
\leq
C
\langle\tau\rangle^{-1+2\delta}
\langle\!\langle u(\tau)\rangle\!\rangle
\biggl(
\sum_{k=1}^3
L(u_k(\tau))
\biggr)
L(u_3(\tau))\nonumber\\
&
\hspace{0.1cm}
+
C
\langle\tau\rangle^{-(3/2)+4\delta}
\langle\!\langle u(\tau)\rangle\!\rangle
\bigl(
\langle\tau\rangle^{-\delta}
{\mathcal N}_4(u(\tau))
\bigr)
N_4(u_3(\tau))\nonumber\\
&
\hspace{0.1cm}
+
C
\langle\tau\rangle^{-1+\eta+2\delta}
\langle\!\langle u(\tau)\rangle\!\rangle
\biggl(
\sum_{i=1,2}
G(u_i(\tau);1)
+
G(u_3(\tau);c_0)
\biggr)
N_4(u_3(\tau))\nonumber
\end{align}
in the same way as in (\ref{2019july291431}) and 
(\ref{2019july291507}). 
Moreover, as in (\ref{2019july291431}) 
and (\ref{2019july291525}), 
we obtain
\begin{align}
J_{33}
\leq&
C
\langle\tau\rangle^{-1+\delta}
\langle\!\langle u(\tau)\rangle\!\rangle
\biggl(
\sum_{k=2}^3
L(u_k(\tau))
\biggr)
L(u_3(\tau))
\\
&
+
\langle\tau\rangle^{-1+2\delta}
\langle\!\langle u(\tau)\rangle\!\rangle
\bigl(
\langle\tau\rangle^{-\delta}
N_4(u_2(\tau))
+
\langle\tau\rangle^{-\delta}
N_4(u_3(\tau))
\bigr)
\langle\tau\rangle^{-\delta}
N_4(u_3(\tau)).\nonumber
\end{align}
For $J_{34}$, we easily obtain
\begin{equation}
J_{34}
\leq
C
\langle\tau\rangle^{-2+5\delta}
\langle\!\langle u(\tau)\rangle\!\rangle^2
\bigl(
\langle\tau\rangle^{-\delta}
{\mathcal N}_4(u(\tau))
+
{\mathcal N}_3(u(\tau))
\bigr)
\bigl(
\langle\tau\rangle^{-\delta}
N_4(u_3(\tau))
\bigr).
\end{equation}
Recall the notation $c_1=c_2=1$, $c_3=c_0$. 
Summing yields for $|a|\leq 3$
\begin{align}\label{2019aug61029}
\langle&t\rangle^{-2\delta}
E(Z^a u_3(t);c_0)
+
\langle t\rangle^{-2\delta}
\int_0^t
G(u_3(\tau);c_0)^2
d\tau\\
&
\leq
C
E(Z^a u_3(0);c_0)\nonumber\\
&
+
C
\langle\!\langle u\rangle\!\rangle_T
\int_0^t
\langle\tau\rangle^{-1+2\delta}
\biggl(
\sum_{k=1}^3
L(u_k(\tau))
\biggr)
L(u_3(\tau))
d\tau\nonumber\\
&
+
C
\langle\!\langle u\rangle\!\rangle_T
\biggl(
\sup_{0<t<T}
\langle t\rangle^{-\delta}
{\mathcal N}_4(u(t))
\biggr)
\int_0^t
\langle\tau\rangle^{-1+\eta+3\delta}
\biggl(
\sum_{i=1}^3
G(u_i(\tau);c_i)
\biggr)
d\tau\nonumber\\
&
+
C
\langle\!\langle u\rangle\!\rangle_T
\biggl(
\sup_{0<t<T}
\langle t\rangle^{-\delta}
{\mathcal N}_4(u(t))
\biggr)^2\nonumber\\
&
+
C
\langle\!\langle u\rangle\!\rangle_T^2
\biggl(
\sup_{0<t<T}
\langle t\rangle^{-\delta}
{\mathcal N}_4(u(t))
+
\sup_{0<t<T}
{\mathcal N}_3(u(t))
\biggr)
\sup_{0<t<T}
\langle t\rangle^{-\delta}
{\mathcal N}_4(u(t)).\nonumber
\end{align}
Now we are in a position to complete the proof of 
Proposition \ref{2019Nov17OK}. 
It is obvious that 
the estimate (\ref{2019Nov17LowEnergy}) 
follows from (\ref{2019aug201631}), (\ref{2019aug201638}), 
and (\ref{2019aug201641}). 
The high energy estimate 
(\ref{2019Nov17HighEnergy}) 
is a direct consequence of (\ref{2019aug61027}), 
(\ref{2019aug61028}), and (\ref{2019aug61029}). 
We have finished the proof. $\hfill\Box$
\section{$L^2$ weighted space-time estimates}\label{l2weighted}
The purpose of this section is to prove the following 
a priori estimates:
\begin{proposition}
The smooth local $($in time$)$ solution $u=(u_1,u_2,u_3)$ to 
$(\ref{eq1})$--$(\ref{data})$ 
defined in $(0,T)\times{\mathbb R}^3$ for some $T>0$ 
satisfies the following a priori estimates for all $t\in (0,T):$
\begin{align}\label{2019aug201643}
\langle&t\rangle^{-(1/2)-4\delta}
\int_0^t
L(u_1(\tau))^2d\tau
\\
&
\leq
C
\sum_{|a|\leq 3}
\|
(\partial Z^a u_1)(0)
\|_{L^2({\mathbb R}^3)}^2\nonumber\\
&
\hspace{0.1cm}
+
C
\langle\!\langle u\rangle\!\rangle_T
\int_0^t
\langle\tau\rangle^{-1+2\delta}
\biggl(
\sum_{k=1}^3
L(u_k(\tau))
\biggr)
L(u_1(\tau))
d\tau\nonumber\\
&
\hspace{0.1cm}
+
C
\langle\!\langle u\rangle\!\rangle_T
\biggl(
\sup_{0<t<T}
\langle t\rangle^{-\delta}
{\mathcal N}_4(u(t))
\biggr)
\int_0^t
\langle\tau\rangle^{-1+\eta+4\delta}
G(u_1(\tau);1)
d\tau\nonumber\\
&
\hspace{0.1cm}
+
C
\langle\!\langle u\rangle\!\rangle_T
\biggl(
\sup_{0<t<T}
\langle t\rangle^{-\delta}
{\mathcal N}_4(u(t))
\biggr)^2\nonumber\\
&
\hspace{0.1cm}
+
C
\langle\!\langle u\rangle\!\rangle_T^2
\biggl(
\sup_{0<t<T}
\langle t\rangle^{-\delta}
{\mathcal N}_4(u(t))
+
\sup_{0<t<T}
{\mathcal N}_3(u(t))
\biggr)
\sup_{0<t<T}
\langle t\rangle^{-\delta}
{\mathcal N}_4(u(t)),\nonumber
\end{align}
\begin{align}\label{2019aug201644}
\langle&t\rangle^{-(1/2)-2\delta}
\int_0^t
L(u_2(\tau))^2
d\tau\\
&
\leq
C
\sum_{|a|\leq 3}
\|
(\partial Z^a u_2)(0)
\|_{L^2({\mathbb R}^3)}^2\nonumber\\
&
\hspace{0.1cm}
+
C
\langle\!\langle u\rangle\!\rangle_T
\int_0^t
\biggl(
\sum_{k=1}^3
L(u_k(\tau))
\biggr)
L(u_2(\tau))d\tau\nonumber\\
&
\hspace{0.1cm}
+
C
\langle\!\langle u\rangle\!\rangle_T
\biggl(
\sup_{0<t<T}
\langle t\rangle^{-\delta}
{\mathcal N}_4(u(t))
\biggr)
\int_0^t
\langle\tau\rangle^{-1+\eta+3\delta}
\biggl(
\sum_{i=1,2}
G(u_i(\tau);1)
\biggr)
d\tau\nonumber\\
&
\hspace{0.1cm}
+
C
\langle\!\langle u\rangle\!\rangle_T
\biggl(
\sup_{0<t<T}
\langle t\rangle^{-\delta}
{\mathcal N}_4(u(t))
\biggr)^2\nonumber\\
&
\hspace{0.1cm}
+
C
\langle\!\langle u\rangle\!\rangle_T^2
\biggl(
\sup_{0<t<T}
\langle t\rangle^{-\delta}
{\mathcal N}_4(u(t))
+
\sup_{0<t<T}
{\mathcal N}_3(u(t))
\biggr)
\sup_{0<t<T}
\langle t\rangle^{-\delta}
{\mathcal N}_4(u(t)),\nonumber
\end{align}
\begin{align}\label{2019aug201645}
\langle&t\rangle^{-(1/2)-2\delta}
\int_0^t
L(u_3(\tau))^2d\tau\\
&
\leq
C
\sum_{|a|\leq 3}
\|
(\partial Z^a u_3)(0)
\|_{L^2({\mathbb R}^3)}^2\nonumber\\
&
\hspace{0.1cm}
+
C
\langle\!\langle u\rangle\!\rangle_T
\int_0^t
\biggl(
\sum_{k=1}^3
L(u_k(\tau))
\biggr)
L(u_3(\tau))
d\tau\nonumber\\
&
\hspace{0.1cm}
+
C
\langle\!\langle u\rangle\!\rangle_T
\biggl(
\sup_{0<t<T}
\langle t\rangle^{-\delta}
{\mathcal N}_4(u(t))
\biggr)
\int_0^t
\langle\tau\rangle^{-1+\eta+3\delta}
\biggl(
\sum_{i=1}^3
G(u_i(\tau);c_i)
\biggr)
d\tau\nonumber\\
&
\hspace{0.1cm}
+
C
\langle\!\langle u\rangle\!\rangle_T
\biggl(
\sup_{0<t<T}
\langle t\rangle^{-\delta}
{\mathcal N}_4(u(t))
\biggr)^2\nonumber\\
&
\hspace{0.1cm}
+
C
\langle\!\langle u\rangle\!\rangle_T^2
\biggl(
\sup_{0<t<T}
\langle t\rangle^{-\delta}
{\mathcal N}_4(u(t))
+
\sup_{0<t<T}
{\mathcal N}_3(u(t))
\biggr)
\sup_{0<t<T}
\langle t\rangle^{-\delta}
{\mathcal N}_4(u(t)).\nonumber
\end{align}
\end{proposition}
In (\ref{2019aug201645}), we have used the notation 
$c_1=c_2=1$, $c_3=c_0$. 
The proof of this proposition naturally uses Lemma \ref{ksstype} 
with $\mu=1/4$. 
With the simple inequality 
$r^{2\mu}\langle r\rangle^{-2\mu}\leq 1$, 
the contributions from the term
$$
\int_0^T\!\!\!\int_{{\mathbb R}^3}
\frac{|w||\Box_c w|}{r^{1-2\mu}\langle r\rangle^{2\mu}}
dxdt
$$
(see the right-hand side of (\ref{l2spacetime})) 
can be handled with use of the Hardy inequality 
or the norm (\ref{weightlocalenergynorm}), 
and therefore the proof is essentially the same as that of 
(\ref{2019aug61027}), (\ref{2019aug61028}), 
and (\ref{2019aug61029}). 
We may omit the details. $\hfill\Box$
\section{Proof of Theorem \ref{ourmaintheorem}}
Now we are ready to complete the proof of 
Theorem \ref{ourmaintheorem} by using the method of 
continuity. 
By the standard contraction-mapping argument, 
it is easy to show that 
for any smooth, compactly supported data (\ref{data}), 
there exists ${\hat T}>0$ 
depending on $\|(f,g)\|_D$ such that 
the equation (\ref{eq1}) admits 
a unique local (in time) solution 
$u=(u_1,u_2,u_3)$ defined in 
the strip $(0,{\hat T})\times{\mathbb R}^3$ 
satisfying 
$\partial_\alpha Z^a u_i\in
C([0,{\hat T});L^2({\mathbb R}^3))$ 
($\alpha=0,1,2,3$, $|a|\leq 3$, $i=1,2,3$) 
and 
${\rm supp}\,u_i(t,\cdot)
\subset
\{x\in{\mathbb R}^3:|x|<R+c^*t\}$ 
$(i=1,2,3,\,0<t<{\hat T})$. 
Here we have set $c^*:=\max\{1,c_0\}$ 
(see (\ref{eq1}) for $c_0$) 
and chosen $R>0$ so that 
${\rm supp}\,f_i\cup{\rm supp}\,g_i
\subset\{x\in{\mathbb R}^3:|x|<R\}$, 
$i=1,2,3$. Actually, this solution is smooth 
in the strip $(0,{\hat T})\times{\mathbb R}^3$, 
and it has the important properties
\begin{align}
&
N_\mu(u_1(t)),\,
N_\mu(u_2(t)),\,
N_\mu(u_3(t))
\in
C([0,{\hat T})),\,\,\mu=3,4,\label{2019property1}\\
&
N_4(u_1(0))
+
N_4(u_2(0))
+
N_4(u_3(0))
\leq
C_d\|(f,g)\|_D\label{2019property2}
\end{align}
for a suitable constant $C_d>0$. 
We employ the numerical constant $C_{61}$ 
appearing in (\ref{2019aug171715}) and set 
\begin{equation}\label{c*definition}
C^*:=
\max
\biggl\{
2C_d,\,
\frac{2}{3}\sqrt{\frac{4}{3}C_{61}}
\biggr\}
\quad\mbox{so that}\quad
\sqrt{\frac{4}{3}C_{61}}
\leq
\frac{3}{2}C^*.
\end{equation}
On the basis of the properties 
(\ref{2019property1})--(\ref{2019property2}), 
for the smooth data (\ref{data}) 
with the support contained in the ball 
$\{x\in{\mathbb R}^3:|x|<R\}$, 
we can define the non-empty set 
of all the numbers $T>0$ such that 
there exists a unique smooth solution $u$ to 
(\ref{eq1})--(\ref{data}) defined in $(0,T)\times{\mathbb R}^3$ 
satisfying 
\begin{align}
&
\langle t\rangle^{-\delta}
{\mathcal N}_4(u(t))
+
{\mathcal N}_3(u(t))
\leq
2C^*\|(f,g)\|_D,\label{n4n3estimate}\\
&
\bigcup_{i=1}^3\,{\rm supp}\,u_i(t,\cdot)
\subset
\{x\in{\mathbb R}^3:|x|<R+c^*t\}\label{finitespeedpropagation}
\end{align}
for all $t\in(0,T)$. 
We define $T^*\in(0,\infty]$ 
as the supremum of this non-empty set. 

To proceed, we assume 
\begin{align}\label{fgsizecondition}
\|(f,g)\|_D
<\varepsilon_0:=\min
\biggl\{
1,\,\frac{1}{8C^*C_{33}C_{60}},\,&
\frac{1}{12C^*C_{60}(C_{31}+2C^*C_{32}C_{60})},\\
&
\frac{1}{2C^*C_{60}C_{62}},\,
\frac{1}{C^*C_{60}C_{63}}
\biggr\}.\nonumber
\end{align}
For the constants appearing above, 
see (\ref{mninequality2019aug16}), 
(\ref{<<u>>smallerthandata}), and (\ref{2019aug171715}). 
We prove
\begin{proposition}\label{propmninequality2019aug17}
Let $u$ be the smooth solution to 
$(\ref{eq1}){\rm -}(\ref{data})$ 
satisfying 
$(\ref{n4n3estimate})$ and 
$(\ref{finitespeedpropagation})$ for all $t\in (0,T^*)$. 
The estimate
\begin{equation}\label{2019mninequality}
{\mathcal M}_\mu(u(t))
\leq
C
{\mathcal N}_\mu(u(t)),
\quad
0<t<T^*
\end{equation}
holds for $\mu=3,4$, 
provided that 
$\|(f,g)\|_D$ satisfies $(\ref{fgsizecondition})$. 
\end{proposition}
\begin{proof}
We proceed closely following the proof of \cite[Proposition 8.1]{HZ2019}. 
When the initial data is identically zero and hence 
the corresponding solution identically vanishes, 
we obviously have (\ref{2019mninequality}). 
We may therefore suppose 
without loss of generality that 
the smooth initial data is not identically zero. 
We then have 
${\mathcal N}_\mu(u(0))>0$. 
Moreover, 
we see ${\mathcal N}_\mu(u(t))>0$ 
for all $t\in (0,T^*)$ 
by repeating basically the same argument 
as in the proof of 
Proposition 8.1 in \cite{HZ2019}. 
(While the uniqueness theorem of 
$C^2$-solutions of John \cite{John1981}, \cite{John1990} 
was employed in \cite{HZ2019}, 
the uniqueness of $H^3\times H^2$-solutions, 
which can be shown in the standard way for such systems 
of semilinear equations as (\ref{eq1}), 
suffices in the present case.) 
Therefore, we may suppose without loss of generality 
that ${\mathcal N}_\mu(u(t))>0$ 
for all $t\in [0,T^*)$. 

Next, we remark the important fact that 
${\mathcal M}_\mu(u(t))$ is continuous on the interval 
$[0,T^*)$. 
This can be easily verified 
thanks to the fact that 
the smooth solution $u$ satisfies (\ref{finitespeedpropagation}) 
on the interval $[0,T^*)$ and hence 
the uniform continuity of 
$\partial_\alpha\partial_x Z^a u_i$ 
($|a|\leq \mu-2, \alpha=0,\dots,3$) 
in such a bounded and closed set as 
$\{(t,x):t\in[0,T+\delta],\,|x|\leq
R+c^*t\}$ 
($\delta$ is a suitable positive constant) 
can be utilized in order to show 
the continuity of ${\mathcal M}_\mu(u(t))$ 
at $t=T\in[0,T^*)$. 
This is the place where our proof 
of Theorem \ref{ourmaintheorem} relies upon the compactness 
of the support of data. 
Since all the constants appearing in our argument are independent of $R$, 
this condition on the support can be actually removed in the standard way.

Now we are ready to prove (\ref{2019mninequality}). 
We start with the inequality 
$$
{\mathcal M}_\mu(u(t))|_{t=0}\leq
C_{KS}{\mathcal N}_\mu(u(t))|_{t=0}
$$ 
for the constant $C_{KS}$ appearing (\ref{mninequality2019aug16}), 
which is a direct consequence of (\ref{KSineq}). 
(See the second term on the right-hand side of (\ref{KSineq}), 
which vanishes at $t=0$.) 
Since $\bigl({\mathcal M}_\mu(u(t))/{\mathcal N}_\mu(u(t))\bigr)|_{t=0}
\leq C_{KS}$ and 
${\mathcal M}_\mu(u(t))/{\mathcal N}_\mu(u(t))$ 
is continuous on the interval $[0,T^*)$, 
we have ${\mathcal M}_\mu(u(t))/{\mathcal N}_\mu(u(t))
\leq 2C_{KS}$, that is 
\begin{equation}\label{2019m2cksn}
{\mathcal M}_\mu(u(t))\leq 2C_{KS}{\mathcal N}_\mu(u(t))
\end{equation}
at least for a short time interval, say, 
$[0,{\tilde T}]\subset [0,T^*)$. 
It remains to show that (\ref{2019m2cksn}) actually 
holds for {\it all} $t\in[0,T^*)$. 
Let
\begin{align}
{\bar T}:=
\sup
\{\,T\in(0,T^*)\,:
&
\,
{\mathcal M}_\mu(u(t))\leq 2C_{KS}{\mathcal N}_\mu(u(t))\\
&
\qquad\quad
(\mu=3,4)\,\mbox{for all}\,\,t\in[0,T)
\}\nonumber
\end{align}
By definition, we know ${\bar T}\leq T^*$. 
To show ${\bar T}=T^*$, 
we proceed as follows. 
By (\ref{<<u>>}), Lemmas \ref{someinequalities2019aug17}
--\ref{traceinequality2019aug17}, and (\ref{n4n3estimate}), 
we get for $t\in (0,{\bar T})$
\begin{align}\label{<<u>>smallerthandata}
\langle\!\langle u(t)\rangle\!\rangle
&
\leq
C
\langle t\rangle^{-\delta}
\bigl(
{\mathcal N}_4(u(t))
+
{\mathcal M}_4(u(t))
\bigr)
+
C
\bigl(
{\mathcal N}_3(u(t))
+
{\mathcal M}_3(u(t))
\bigr)\\
&
\leq
C_{60}
\bigl(
\langle t\rangle^{-\delta}
{\mathcal N}_4(u(t))
+
{\mathcal N}_3(u(t))
\bigr)
\leq
2C^*C_{60}
\|(f,g)\|_D.\nonumber
\end{align}
Here, $C_{60}$ is a suitable positive constant. 
Owing to the size condition (\ref{fgsizecondition}), 
Proposition \ref{2019june25mnineq} 
combined with the last inequality (\ref{<<u>>smallerthandata}) 
immediately yields for $\mu=3,4$
\begin{equation}\label{mn32inequality}
{\mathcal M}_\mu(u(t))
\leq
\frac32
C_{KS}
{\mathcal N}_\mu(u(t)),
\quad
0<t<{\bar T}.
\end{equation}
Since ${\mathcal M}_\mu(u(t))/{\mathcal N}_\mu(u(t))$ 
is continuous on the interval $[0,T^*)$, 
we have finally arrived at the conclusion 
${\bar T}=T^*$. 
Indeed, if we assume ${\bar T}<T^*$, 
then the estimate (\ref{mn32inequality}) 
contradicts the definition of ${\bar T}$. 
We have finished the proof of 
Proposition \ref{propmninequality2019aug17}. 
\end{proof}
Now we are going to prove the crucial a priori estimate
\begin{equation}\label{2019crucial1443}
\langle t\rangle^{-\delta}
{\mathcal N}_4(u(t))
+
{\mathcal N}_3(u(t))
\leq
\frac32
C^*\|(f,g)\|_D,
\quad
0<t<T^*.
\end{equation}
This estimate combined with the standard local existence theorem 
will immediately implie $T^*=\infty$, i.e., global existence. 
Just for simplicity, we use the notation
\begin{align*}
&
{\mathcal G}(t)
:=
\langle t\rangle^{-\delta}
\|
G(u_1(\cdot);1)
\|_{L^2((0,t))}
+
\|
G(u_2(\cdot);1)
\|_{L^2((0,t))}
+
\|
G(u_3(\cdot);c_0)
\|_{L^2((0,t))},\\
&
{\mathcal L}(t)
:=
\langle t\rangle^{-(1/4)}
\bigl(
\langle t\rangle^{-\delta}
\|
L(u_1(\cdot))
\|_{L^2((0,t))}
+
\|
L(u_2(\cdot))
\|_{L^2((0,t))}
+
\|
L(u_3(\cdot))
\|_{L^2((0,t))}
\bigr).
\end{align*}
Without loss of generality, we may suppose $T^*>1$ 
because we are considering solutions with small data. 
It then follows from (\ref{2019Nov17LowEnergy}), 
(\ref{2019Nov17HighEnergy}), 
(\ref{2019aug201643}), 
(\ref{2019aug201644}), and 
(\ref{2019aug201645}) 
that for any $T$ with $1<T<T^*$ we have
\begin{align}\label{2019aug171715}
\biggl(&
\sup_{0<t<T}
\langle t\rangle^{-\delta}
{\mathcal N}_4(u(t))
+
\sup_{0<t<T}
{\mathcal N}_3(u(t))
\biggr)^2\\
&
+
\biggl(
\sup_{0<t<T}
\langle t\rangle^{-\delta}
{\mathcal G}(t)
\biggr)^2
+
\biggl(
\sup_{0<t<T}
\langle t\rangle^{-\delta}
{\mathcal L}(t)
\biggr)^2\nonumber\\
&
\leq
C_{61}
\|(f,g)\|_D^2
+
C_{62}
\langle\!\langle u\rangle\!\rangle_T
\biggl(
\sup_{0<t<T}
\langle t\rangle^{-\delta}
{\mathcal L}(t)
\biggr)^2\nonumber\\
&
+
C_{63}
\langle\!\langle u\rangle\!\rangle_T
\biggl(
\sup_{0<t<T}
\langle t\rangle^{-\delta}
{\mathcal N}_4(u(t))
\biggr)
\biggl(
\sup_{0<t<T}
\langle t\rangle^{-\delta}
{\mathcal G}(t)
\biggr)\nonumber\\
&
+
C_{64}
\langle\!\langle u\rangle\!\rangle_T
\biggl(
\sup_{0<t<T}
\langle t\rangle^{-\delta}
{\mathcal N}_4(u(t))
+
\sup_{0<t<T}
{\mathcal N}_3(u(t))
\biggr)^2.\nonumber
\end{align}
Here the positive constants $C_{6i}$ $(i=1,\dots,4)$ are independent 
of $T$. 
We note that $\delta$ and $\eta$ are so small that 
the idea of decomposing the interval $[1,T]$ dyadically 
has played an important role as in such previous papers 
as \cite[p.\,363]{Sogge2003}, \cite[(122)--(125)]{HZ2019}. 
For any $T$ with $T<T^*$, 
we easily see  
$$
\sup_{0<t<T}
\langle t\rangle^{-\delta}
{\mathcal G}(t),
\quad
\sup_{0<t<T}
\langle t\rangle^{-\delta}
{\mathcal L}(t)
<\infty
$$
and it is therefore possible 
to move 
the second and the third terms 
on the right-hand side of (\ref{2019aug171715}) 
to its left-hand side. 
Using the estimate (\ref{<<u>>smallerthandata}), 
which holds for all $t\in(0,T^*)$, 
and (\ref{fgsizecondition}), 
we thereby obtain 
\begin{align}
\biggl(&
\sup_{0<t<T}
\langle t\rangle^{-\delta}
{\mathcal N}_4(u(t))
+
\sup_{0<t<T}
{\mathcal N}_3(u(t))
\biggr)^2\\
&
\leq
C_{61}
\|(f,g)\|_D^2\nonumber\\
&
\hspace{0.1cm}
+
\biggl(
\frac{1}{2}C_{63}+C_{64}
\biggr)
\langle\!\langle u\rangle\!\rangle_T
\biggl(
\sup_{0<t<T}
\langle t\rangle^{-\delta}
{\mathcal N}_4(u(t))
+
\sup_{0<t<T}
{\mathcal N}_3(u(t))
\biggr)^2,\nonumber
\end{align}
which immediately implies
\begin{equation}
\frac34
\biggl(
\sup_{0<t<T}
\langle t\rangle^{-\delta}
{\mathcal N}_4(u(t))
+
\sup_{0<t<T}
{\mathcal N}_3(u(t))
\biggr)^2
\leq
C_{61}
\|(f,g)\|_D^2
\end{equation}
thanks to (\ref{<<u>>smallerthandata}) and (\ref{fgsizecondition}). 
Since $T(<T^*)$ is arbitrary and 
the constant $C_{61}$ is independent of $T$, 
we finally obtain 
\begin{equation}
\sup_{0<t<T^*}
\langle t\rangle^{-\delta}
{\mathcal N}_4(u(t))
+
\sup_{0<t<T^*}
{\mathcal N}_3(u(t))
\leq
\sqrt{\frac{4}{3}C_{61}}\|(f,g)\|_D
\leq
\frac32
C^*
\|(f,g)\|_D.
\end{equation}
See (\ref{c*definition}). 
Now we are in a position to show $T^*=\infty$. 
Assume $T^*<\infty$. 
By solving (\ref{eq1}) with data 
$(u_i(T^*-\delta,x),(\partial_t u_i)(T^*-\delta,x))
\in C_0^\infty({\mathbb R}^3)\times C_0^\infty({\mathbb R}^3)$ 
given at $t=T^*-\delta$ ($\delta$ is a sufficiently small 
positive constant), 
we can extend the local solution under consideration 
smoothly to a larger strip, 
say, 
$\{(t,x):\,0<t<{\tilde T},\,x\in{\mathbb R}^3\}$, 
where $T^*<{\tilde T}$. 
The local solution thereby extended satisfies 
\begin{align*}
&
N_\mu(u_1(t)),\,
N_\mu(u_2(t)),\,
N_\mu(u_3(t))
\in
C([0,{\tilde T})),\,\,\mu=3,4,\\
&
\bigcup_{i=1}^3\,{\rm supp}\,u_i(t,\cdot)
\subset
\{x\in{\mathbb R}^3:|x|<R+c^*t\},\quad
0<t<{\tilde T}.
\end{align*}
Since 
$\bigl(
\langle t\rangle^{-\delta}
{\mathcal N}_4(u(t))
+
{\mathcal N}_3(u(t))
\bigr)|_{t=T^*}
\leq (3/2)C^*\|(f,g)\|_D$
by (\ref{2019crucial1443}) 
and 
$
\langle t\rangle^{-\delta}
{\mathcal N}_4(u(t))
+
{\mathcal N}_3(u(t))
\in 
C([0,{\tilde T}))
$, 
we see that there exists $T'\in (T^*,{\tilde T}]$ 
such that 
$
\langle t\rangle^{-\delta}
{\mathcal N}_4(u(t))
+
{\mathcal N}_3(u(t))
\bigr)
\leq 2C^*\|(f,g)\|_D$ 
for all 
$t\in (0,T')$, 
which contradicts the definition of $T^*$. 
Hence we have $T^*=\infty$. 
We have finished the proof.$\hfill\Box$

%
\end{document}